\documentclass{amsart}
\usepackage{amsfonts,amssymb}

\usepackage{amsmath}
\usepackage{amsthm}
\usepackage{amsrefs}
\usepackage{qsymbols}
\usepackage{latexsym}
\usepackage{chngcntr}
\usepackage[noadjust]{cite}
\usepackage{paralist}
\usepackage{esint}

\newtheorem{theorem}{Theorem}[section]

\newtheorem{lemma}[theorem]{Lemma}
\newtheorem{proposition}[theorem]{Proposition}
\newtheorem{corollary}[theorem]{Corollary}
\theoremstyle{definition}
\newtheorem{definition}[theorem]{Definition}
\newtheorem{remark}[theorem]{Remark}

\newtheorem{assumption}[theorem]{Assumption}
\newtheorem{agreement}[theorem]{Agreement}

\newcommand{\IR}{\mathbb{R}}
\newcommand{\IC}{\mathbb{C}}
\newcommand{\IN}{\mathbb{N}}

\newcommand{\Lop}{{\mathcal{L}}}
\renewcommand{\S}{{\mathrm{S}}}

\newcommand{\m}{\mathrm{m}}
\newcommand{\HH}{\mathcal{H}}
\newcommand{\V}{\mathcal{V}}
\renewcommand{\a}{\mathfrak{a}}

\newcommand{\g}{\Gamma}
\newcommand{\gb}{\Gamma_B}
\newcommand{\gbs}{\Gamma_B^*}
\newcommand{\pb}{\Pi_B}

\newcommand{\prmeas}{\|\gamma_t(x)\|_{\Lop(\IC^N)}^2}
\newcommand{\boxmeas}{\frac{\d x \, \d t}{t}}
\newcommand{\NN}{\mathfrak{N}}
\renewcommand{\L}{\mathrm{L}}
\renewcommand{\H}{\mathrm{H}}

\newcommand{\C}{\mathrm{C}}

\newcommand{\E}{{\mathfrak{E}}}

\renewcommand{\d}{\mathrm{d}}
\renewcommand{\i}{\mathrm{i}}

\newcommand{\eps}{\varepsilon}
\newcommand{\ind}{{\mathbf{1}}}
\newcommand{\z}{z}
\newcommand{\argdot}{\cdot}
\newcommand{\abs}[1]{\lvert#1\rvert}
\newcommand{\cl}[1]{\overline{#1}}
\newcommand{\bd}{\partial}
\DeclareMathOperator{\supp}{supp}
\DeclareMathOperator{\dist}{d}
\DeclareMathOperator{\diam}{diam}

\renewcommand{\Re}{\operatorname{Re}}

\DeclareMathOperator{\Id}{Id}
\newcommand{\Rg}{\mathcal{R}}
\newcommand{\Ke}{\mathcal{N}}
\newcommand{\dom}{\mathcal{D}}

\numberwithin{equation}{section} 
\makeatletter
\newcommand{\labitem}[2]{%
\def\@itemlabel{#1}
\item[\normalfont(#1)]
\def\@currentlabel{#1}\label{#2}}
\makeatother
\title[Kato follows from an Extrapolation Property of the Laplacian]{The Kato Square Root Problem follows from an Extrapolation Property of the Laplacian}
\author{Moritz Egert, Robert Haller-Dintelmann, and Patrick Tolksdorf}
\address{Fachbereich Mathematik, Technische Universit\"at Darmstadt, Schlossgartenstr. 7, 64289 Darmstadt, Germany}
\email{egert@mathematik.tu-darmstadt.de}
\email{haller@mathematik.tu-darmstadt.de}
\email{tolksdorf@mathematik.tu-darmstadt.de}
\thanks{The first and the third author were supported by ``Studienstiftung des deutschen Volkes''.}

\date{February 27, 2015}

\begin{document}
\begin{abstract}
On a domain $\Omega \subseteq \IR^d$ we consider second-order elliptic systems in divergence-form with bounded complex coefficients, realized via a sesquilinear form with domain $\H_0^1(\Omega) \subseteq \V \subseteq \H^1(\Omega)$. Under very mild assumptions on $\Omega$ and $\V$ we show that the solution to the Kato Square Root Problem for such systems can be deduced from a regularity result for the fractional powers of the negative Laplacian in the same geometric setting. This extends earlier results of M\textsuperscript{c}Intosh \cite{KatoMultiplier} and Axelsson-Keith-M\textsuperscript{c}Intosh \cite{AKM} to non-smooth coefficients and domains.
\end{abstract}
\maketitle

\section{Introduction}
\label{Sec: Introduction}

\noindent We consider a second-order $m \times m$ elliptic system
\begin{align*}
 Au = -\sum_{\alpha, \beta = 1}^d \partial_\alpha(a_{\alpha, \beta} \partial_{\beta} u)
\end{align*}
in divergence-form with bounded $\IC^{m \times m}$-valued coefficients $a_{\alpha, \beta}$ on a domain $\Omega \subseteq \IR^d$. As usual, $A$ is interpreted as a maximal accretive operator on $\L^2(\Omega)$ via a sesquilinear form defined on some closed subset $\V$ of $\H^1(\Omega)$ that contains $\H_0^1(\Omega)$. A fundamental question due to Kato \cite{Katos-Conjecture} and refined by Lions \cite{Counterexample_Lions}, having made history as the \emph{Kato Square Root Problem}, is whether $A$ has the \emph{square root property} $\dom(\sqrt{A}) = \V$, i.e.\ whether the domain of the maximal accretive square root of $A$ coincides with the form domain. 

Whereas for self-adjoint $A$ this is immediate from abstract form theory \cite{Kato}, the full problem remained open for almost 40 years. It were Auscher, Hofmann, Lacey, M\textsuperscript{c}Intosh, and Tchamitchian, who eventually gave a proof on $\Omega = \IR^d$ exploiting the full strength of harmonic analysis \cite{Kato-Square-Root-Proof, Kato-Systems-Proof}. Shortly after, Auscher and Tchamitchian used localization techniques to solve the Kato Square Root Problem on strongly Lipschitz domains $\Omega$ complemented by either pure Dirichlet or pure Neumann boundary conditions \cite{AT2}. These refer to the cases $\V = \H_0^1(\Omega)$ and $\V = \H^1(\Omega)$. For a survey we refer to \cite{McIntosh-KatoSurvey, Kato-Square-Root-Proof} and the references therein.

A milestone toward general form domains has then been set by Axelsson, Keith, and M\textsuperscript{c}Intosh \cite{AKM-QuadraticEstimates, AKM}, who introduced an operator theoretic framework that allows to cast the Kato Square Root Problem for almost arbitrary $\Omega$ and $\V$ as an abstract first-order problem. By these means they gave a solution if $\Omega$ is a smooth domain, $D$ is a smooth part of the boundary $\bd \Omega$, and $\V$ is the subspace of $\H^1(\Omega)$ containing those functions that vanish on $D$ -- and moreover for global bi-Lipschitz images of these configurations \cite{AKM}.

Much earlier, in 1985 M\textsuperscript{c}Intosh revealed another profound structural aspect of the Kato Square Root Problem: Assuming some smoothness on the coefficients and the domain $\Omega$, he proved that on arbitrary form domains $\V$ the affirmative answer to Kato's problem follows if the square root property for the easiest elliptic differential operator -- the self-adjoint negative Laplacian -- can be extrapolated to fractional powers of exponent slightly above $\frac{1}{2}$, cf.\ \cite{KatoMultiplier}. A similar approach has been pursued in \cite{AKM}.

Our main result is a reduction theorem in this spirit for second-order elliptic systems whose coefficients are merely bounded. We do so under significantly weaker geometric assumptions than in \cite{AKM} and \cite{KatoMultiplier} but in contrast to \cite{KatoMultiplier} we have to assume that the form domain is invariant under multiplication by smooth functions. As an application we have obtained an extension of previous results on the Kato Square Root Problem for mixed boundary conditions \cite{Darmstadt-KatoMixedBoundary}. The key technique is a \emph{$\pb$-type theorem} in the spirit of \cite{AKM-QuadraticEstimates}, which we state as our second main result and which allows for further applications, e.g.\ to prove well-posedness of boundary value problems on cylindrical domains, see the upcoming work of P.\ ~Auscher and the first author.

The paper is organized as follows. After introducing some notation and the geometric setup in Section~\ref{Sec: Notation and general assumptions}, we state our main results in Section~\ref{Sec: Main results}. The hypotheses underlying our $\pb$-theorem are discussed in Section~\ref{Sec: The abstract framework}. In Section~\ref{Sec: Proof of main result} we deduce our main result from the $\pb$-theorem. For the reader's convenience, necessary tools from functional calculus are recalled beforehand in Section~\ref{Sec: Functional calculi}. In the remaining sections we develop the proof of the $\pb$-theorem. Our argument builds upon the techniques being introduced in \cite{AKM-QuadraticEstimates} as did many other square root type results before \cite{AKM, Morris, Bandara, Bandara2}, but as a novelty allows the presence of a non-smooth boundary. We suggest to keep a copy of \cite{AKM-QuadraticEstimates} handy as duplicated arguments with this paper are omitted.
\section{Notation and General Assumptions}
\label{Sec: Notation and general assumptions}

\noindent Most of our notation is standard. Throughout, the dimension $d \geq 2$ of the underlying Euclidean space is fixed. The open ball in $\IR^d$ with center $x$ and radius $r>0$ is denoted by $B(x,r)$. For abuse of notation we use the symbol $\abs{\argdot}$ for both the Euclidean norm of vectors in $\IC^n$, $n \geq 1$, as well as for the $d$-dimensional Lebesgue measure. For $z \in \IC$ we put $\langle z \rangle := 1 + \abs{z}$. The Euclidean distance between subsets $E$ and $F$ of $\IR^d$ is $\d(E,F)$. If $E = \{x \}$, then the abbreviation $\d(x,E)$ is used. The complex logarithm $\log$ is always defined on its principal branch $\IC \setminus (-\infty,0]$. The indicator function of a set $E \subseteq \IR^d$ is $\ind_E$ and for convenience we abbreviate the maps $z \mapsto 1$ and $z \mapsto z$ by $1$ and $\z$, respectively. For average integrals the symbol $\fint$ is used.

We allow ourselves the freedom to write $a \lesssim b$ if there exists $C > 0$ not depending on the parameters at stake such that $a \leq C b$ holds. Likewise, we use the symbol $\gtrsim$ and we write $a \simeq b$ if both $a \lesssim b$ and $b \lesssim a$ hold.

\subsection{Function spaces} 
\label{Subsec: Function spaces}

The Hilbert space of square integrable, $\IC^n$-valued functions on a Borel set $\Xi \subseteq \IR^d$ is $\L^2(\Xi; \IC^n)$. If $\Xi$ is open, then $\H^1(\Xi; \IC^n)$ is the associated first-order Sobolev space with its usual Hilbertian norm and $\H_0^1(\Omega; \IC^n)$ denotes the $\H^1$-closure of $\C_c^\infty(\Xi ; \IC^n)$, the space of smooth functions with compact support in $\Xi$. The \emph{Bessel potential spaces} with differentiability $s>0$ and integrability $2$ are $\H^{s,2}(\IR^d; \IC^n)$, see \cite[Sec.\ ~2.3.3]{Triebel} and $\H^{s,2}(\Xi; \IC^n):= \{u|_\Xi: u \in \H^{s,2}(\IR^d; \IC^n)\}$ is equipped with the quotient norm
$\|u\|_{\H^{s,2}(\Xi ; \IC^n)}:= \inf \{\|v\|_{\H^{s,2}(\IR^d; \IC^n)}: v = u \text{\, a.e.\ on $\Xi$}\}$.

\subsection{Operators on Hilbert spaces}
\label{Subsec: Operators on Hilbert spaces}

Any Hilbert space $\HH$ under consideration is taken over the complex numbers. Concerning linear operators we follow the standard notation. If $B_1$ and $B_2$ are operators in $\HH$ then $B_1 + B_2$ and $B_1 B_2$ are defined on their natural domains
\begin{align*}
 \dom(B_1 + B_2):= \dom(B_1) \cap \dom(B_2) \quad \text{and} \quad \dom(B_1 B_2):= \{u \in \dom (B_2): B_2 u \in \dom(B_1) \}.
\end{align*}
Their \emph{commutator} is $[B_1, B_2]:= B_1 B_2 - B_2 B_1$.

\subsection{Geometric setup and the elliptic operator}
\label{Subsec: The elliptic operator}

In this section we define the elliptic operator $Au = -\sum_{\alpha, \beta = 1}^d \partial_\alpha(a_{\alpha, \beta} \partial_{\beta} u)$ under consideration properly by means of Kato's form method \cite{Kato}. Starting from now, the codimension $m \geq 1$ -- the number of ``equations'' -- is fixed.

Throughout this work we assume the following geometric setup.

\begin{assumption}
\label{Ass: Geometric setup}
\begin{enumerate}

 \labitem{$\Omega$}{d} We assume that $\Omega \subseteq \IR^d$ is a \emph{$d$-set} in the sense of Jonsson/Wallin~\cite{Jonsson-Wallin}, i.e.\ that it satisfies the \emph{$d$-Ahlfors} or \emph{measure density condition}
\begin{align*}
\hspace{30pt} \abs{\Omega \cap B(x,r)} \simeq r^d \qquad(x \in \Omega,\, 0<r \leq 1). 
\end{align*}

\labitem{$\bd \Omega$}{d-1} We assume that $\bd \Omega$ is a \emph{$(d-1)$-set} in the sense of Jonsson/Wallin~\cite{Jonsson-Wallin}, i.e.\ that it satisfies the \emph{Ahlfors-David condition}
\begin{align*}
 \hspace{30pt} \m_{d-1}(\bd \Omega \cap B(x,r))\simeq r^{d-1} \qquad(x \in \bd \Omega,\, 0<r \leq 1),
\end{align*}
where here and throughout $\m_{d-1}$ denotes the $(d-1)$-dimensional Hausdorff measure.

\labitem{$\V$}{V} We assume that $\V$ is a closed subspace of $\H^1(\Omega; \IC^m)$ that contains $\H_0^1(\Omega; \IC^m)$ and is stable under multiplication by smooth functions in the sense
\begin{align*}
 \varphi \V \subseteq \V \qquad (\varphi \in \C_c^\infty(\IR^d; \IC)).
\end{align*}
Moreover, we assume that $\V$ has the \emph{$\H^1$-extension property}, i.e.\ that there exists a bounded operator $\E: \V \to \H^1(\IR^d; \IC^m)$ such that $\E u = u$ a.e.\ on $\Omega$ for each $u \in \V$.

\labitem{$\alpha$}{alpha} We assume that for some $\alpha \in (0,1)$ the complex interpolation space $[\L^2(\Omega; \IC^m),\V]_\alpha$ coincides with $\H^{\alpha,2}(\Omega; \IC^m)$ and that their norms are equivalent. 
\end{enumerate}
\end{assumption}

Let us comment on these assumptions.

\begin{remark}
\label{Rem: Remark on geometry}
 \begin{enumerate}

  \item The stability assumption on $\V$ is e.g.\ satisfied for the usual choices of $\V$ modeling (mixed) Dirichlet and Neumann boundary conditions \cite{Ouhabaz, Darmstadt-KatoMixedBoundary}. 

  \item The $\H^1$ extension property for $\V$ is trivially satisfied if $\Omega$ admits a bounded Sobolev extension operator $\E: \H^1(\Omega; \IC^m) \to \H^1(\IR^d; \IC^m)$. In this case also \eqref{d} holds \cite[Thm.\ ~2]{Hajlasz-Koskela-Tuominen}.

  \item Assumption $(d-1)$ is common in the treatment of boundary value problems, being among the weakest geometric conditions that allow to define boundary traces, cf.\ \cite{Jonsson-Wallin}.

  \item Assumption $(\alpha)$ should be considered as a geometric one. A common way to force its validity is to assume that $\Omega$ is a Sobolev extension domain and that
  \begin{align}
  \label{McIntosh condition}
  \tag{M\textsuperscript{c}}
    \big[\L^2(\Omega; \IC^m),\H_0^1(\Omega; \IC^m)\big]_\alpha = \big[\L^2(\Omega; \IC^m),\H^1(\Omega; \IC^m)\big]_\alpha
  \end{align}
  holds up to equivalent norms. Indeed, \eqref{alpha} then follows from $\H_0^1(\Omega; \IC^m) \subseteq \V \subseteq \H^1(\Omega; \IC^m)$ and standard interpolation results \cite[Sec.\ 1.2.4/2.4.2]{Triebel}. The condition \eqref{McIntosh condition} has been introduced in this context by M\textsuperscript{c}Intosh \cite{KatoMultiplier}.

  \item Among the vast variety of Sobolev extension domains satisfying \eqref{d-1} and M\textsuperscript{c}Intosh's condition for all $\alpha \in (0, \frac{1}{2})$ are the whole space $\IR^d$ \cite[Sec.\ 2.4.1]{Triebel}, the upper half space $\IR_+^d$ \cite[Sec.\ ~2.10]{Triebel} from which the result for special Lipschitz domains can be deduced, as well as bounded Lipschitz domains \cite[Thm.\ 3.1]{Griepentrog-InterpolationOnGroger}, \cite[Sec.\ 4.3.1]{Triebel}. Assumption~\ref{Ass: Geometric setup} then reduces to the stability assumption on $\V$. However, configurations in which $\Omega$ is not a Sobolev extension domain though \eqref{d}, \eqref{d-1}, \eqref{V}, and \eqref{alpha} are satisfied, naturally occur in the treatment of mixed boundary value problems, cf.\ \cite{Darmstadt-KatoMixedBoundary} and the references therein.
 \end{enumerate}
\end{remark}

Concerning the coefficients of $A$ we make the following standard assumption.

\begin{assumption}
\label{Ass: Ellipticity}
We assume $a_{\alpha, \beta} \in \L^{\infty}(\Omega; \IC^{m \times m})$ for all $1 \leq \alpha, \beta \leq d$ and that the associated sesquilinear form
\begin{align*}
 \a: \V \times \V \to \IC, \quad \a(u,v) = \sum_{\alpha, \beta = 1}^d \int_\Omega a_{\alpha, \beta}(x) \partial_\beta u(x) \cdot \cl{\partial_\alpha v(x)}
\end{align*}
is elliptic in the sense that for some $\lambda > 0$ it satisfies the \emph{G\aa{}rding inequality}
\begin{align}
\label{Garding inequality}
 \Re (\a(u,u)) \geq \lambda \|\nabla u\|_{\L^2(\Omega; \IC^{dm})}^2 \qquad (u \in \V).
\end{align}
\end{assumption}

Since $\V$ is dense in $\L^2(\Omega; \IC^m)$ and $\a$ is elliptic, classical form theory \cite[Ch.\ VI]{Kato} yields that the associated operator $A$ on $\L^2(\Omega; \IC^m)$ given by
\begin{align*}
  \a(u,v) = \langle Au, v\rangle_{\L^2(\Omega; \IC^m)} \qquad (u \in \dom(A),\, v \in \V)
\end{align*}
on
\begin{align*}
 \dom(A) := \big \{u \in \V: \text{$\a(u, \argdot)$ boundedly extends to $\L^2(\Omega; \IC^m)$} \big \}
\end{align*}
is \emph{maximal accretive}. By this we mean that $A$ is closed and for $z$ in the open left complex halfplane $z - A$ is invertible with $\|(z - A)^{-1}\|_{\Lop(\L^2(\Omega; \IC^m))} \leq \abs{\Re (z)}^{-1}$. The choice $a_{\alpha, \beta} = \delta_{\alpha, \beta} \Id_{\IC^{m \times m}}$, where $\delta$ is Kronecker's delta, yields the negative of the (coordinatewise) \emph{weak Laplacian} $\Delta_\V$ with form domain $\V$.

Maximal accretivity allows to define \emph{fractional powers} $(\eps + A)^\alpha$ for all $\alpha, \eps \geq 0$ by means of the functional calculus for sectorial operators, see Section~\ref{Sec: Functional calculi}. The so-defined \emph{square root} $\sqrt{A}$ of $A$ is the unique maximal accretive operator such that $\sqrt{A} \sqrt{A} = A$ holds, cf.\ \cite[Thm.\ V.3.35]{Kato} and \cite[Cor.\ ~7.1.13]{Haase}.

\section{Main Results}
\label{Sec: Main results}

\noindent The main result we want to prove in this paper is the following.

\begin{theorem}
\label{Thm: Main Theorem}
Let Assumptions~\ref{Ass: Geometric setup} and \ref{Ass: Ellipticity} be satisfied and let $\Delta_\V$ be the weak Laplacian with form domain $\V$. If for the \emph{same} $\alpha$ as in Assumption~\ref{Ass: Geometric setup}
\begin{align}
\label{Emb: Embedding condition}
\tag{E}
\dom( (1- \Delta_\V)^{1/2 + \alpha/2}) \subseteq \H^{1 + \alpha,2} (\Omega ; \IC^m)
\end{align}
with continuous inclusion, then $A$ has the square root property
\begin{align*}
\dom(\sqrt{A}) = \dom(\sqrt{1+A}) = \V \quad \text{with} \quad \| (\sqrt{1 + A}) u \|_{\L^2(\Omega ; \IC^m)} \simeq \|  u \|_{\V} \qquad (u \in \V).
\end{align*}
\end{theorem}

By a classical result \cite[Thm.\ VI.2.23]{Kato} the self-adjoint operator $1-\Delta_\V$ has the square root property $\dom(\sqrt{1-\Delta_\V}) = \V \subseteq \H^1(\Omega; \IC^m)$. Hence, our main result may informally be stated as follows:

\medskip

\begin{center}
\begin{minipage}{0.86\textwidth}
 \textit{If the square root property for the negative Laplacian with form domain $\V$ extrapolates to fractional powers with exponent slightly above $\frac{1}{2}$, then every elliptic differential operator in divergence form with form domain $\V$ has the square root property.}
\end{minipage}
\end{center}

\begin{remark}
\label{Rem: Remark to the main result}
\begin{enumerate}

 \item The conditions \eqref{McIntosh condition} and \eqref{Emb: Embedding condition} are those imposed by M\textsuperscript{c}Intosh \cite{KatoMultiplier} to solve the Kato Square Root Problem for operators $A$ with H\"{o}lder continuous coefficients.
 
 \item In applications it usually suffices that \eqref{alpha} and \eqref{Emb: Embedding condition} hold for different choices of $\alpha$ since then, by interpolation, both conditions can be met simultaneously for some possibly smaller value of $\alpha$, cf.\ \cite{Darmstadt-KatoMixedBoundary}.
\end{enumerate}
\end{remark}

In Section~\ref{Sec: The abstract framework} we will deduce Theorem~\ref{Thm: Main Theorem} from the following \emph{$\pb$-theorem}. In fact, Theorem~\ref{Thm: The pibe theorem} is a generalization of the main result in \cite{AKM} to non-smooth domains. For the notion of bisectorial operators see Section~\ref{Sec: Functional calculi}. Corollary~\ref{Cor: Abstract Kato problem} is discussed in more detail at the end of Section~\ref{Sec: Functional calculi}.

\begin{theorem}
\label{Thm: The pibe theorem}
Let $k \in \IN$ and $N=km$. On the Hilbert space $\HH:= (\L^2(\Omega; \IC^m))^k$ consider operators $\g$, $B_1$, and $B_2$ satisfying \eqref{H1} - \eqref{H7}, see Section~\ref{Sec: The abstract framework}. Then the \emph{perturbed Dirac type operator} $\pb:= \g + B_1 \g^* B_2$ is bisectorial of some angle $\omega \in (0,\frac{\pi}{2})$ and satisfies quadratic estimates
\begin{align}
\label{Eq: Quadratic estimate for pb}
 \int_0^\infty \|t \pb (1+t^2 \pb^2)^{-1}u\|_\HH^2 \; \frac{\d t}{t} \simeq \|u\|_\HH^2 \qquad (u \in \cl{\Rg(\pb)}).
\end{align}
Moreover, implicit constants depend on $B_1$ and $B_2$ only through the constants quantified in \eqref{H2}.
\end{theorem}

\begin{corollary}
\label{Cor: Abstract Kato problem}
The part of $\pb$ in $\cl{\Rg(\pb)}$ is an injective bisectorial operator of angle $\omega$ with a bounded $\H^\infty(\S_\psi)$-calculus for each $\psi \in (\omega, \frac{\pi}{2})$. In particular, it shares the \emph{Kato square root type estimate}
\begin{align*}
 \dom({\textstyle \sqrt{\pb^2}}) = \dom(\pb) \quad \text{with} \quad \|{\textstyle \sqrt{\pb^2}} u\|_\HH \simeq \|\pb u\|_\HH \qquad (u \in \dom(\pb)).
\end{align*}
\end{corollary}

\section{Functional Calculi}
\label{Sec: Functional calculi}

\noindent We recall the functional calculi for sectorial and bisectorial operators. For sectorial operators we follow the treatment in \cite[Ch.\ ~2]{Haase}. Good references for the bisectorial case are \cite{Duelli-Thesis}, \cite{Duelli-Weiss}, see also \cite[Ch.~3]{Eigene_Diss}
 
Throughout, given $\varphi \in (0,\pi)$, denote by $\S_\varphi^+:= \{z \in \IC \setminus \{0\}: \abs{\arg{z}} < \varphi \}$ the open \emph{sector} with vertex $0$ and opening angle $2 \varphi$ symmetric around the positive real axis. If $\varphi \in (0, \frac{\pi}{2})$ then $\S_{\varphi}:= \S_\varphi^+ \cup (-\S_\varphi^+)$ is the corresponding open \emph{bisector}. An operator $B$ on a  Hilbert space $\HH$ is \emph{sectorial} of \emph{angle} $\varphi \in (0,\pi)$ if its spectrum is contained in $\cl{\S_\varphi^+}$ and 
\begin{align*}
 \sup \big\{ \|\lambda (\lambda - B)^{-1}\|_{\Lop(\HH)}: \lambda \in \IC \setminus \cl{\S_\psi^+} \big\} < \infty \qquad (\psi \in (\varphi, \pi)).
\end{align*}
Likewise, $B$ is \emph{bisectorial} of \emph{angle} $\varphi \in (0,\frac{\pi}{2})$ if $\sigma(B) \subseteq \cl{\S_\varphi}$ and
\begin{align*}
 \sup \big\{ \|\lambda (\lambda - B)^{-1}\|_{\Lop(\HH)}: \lambda \in \IC \setminus \cl{\S_\psi} \big \} < \infty \qquad (\psi \in (\varphi, \tfrac{\pi}{2})).
\end{align*}
A sectorial or bisectorial operator $B$ on $\HH$ necessarily is densely defined and induces a topological decomposition $\HH = \Ke(B) \oplus \cl{\Rg(B)}$, see \cite[Prop.\ 2.1.1]{Haase} for the sectorial case. The bisectorial case can be treated similarly.

\subsection{Construction of the functional calculi}
\label{Subsec: Construction of the functional calculi}

For an open set $U \subseteq \IC$ denote by $\H^\infty(U)$ the Banach algebra of bounded holomorphic functions on $U$ equipped with the supremum norm $\|\cdot\|_{\infty,U}$ and let
\begin{align*}
\H_0^\infty (U) := \left\{g \in \H^\infty(U) \,\middle|\, \exists \, C,s>0 \; \forall \, z \in U: \abs{g(z)} \leq  C \min\{\abs{z}^s, \abs{z}^{-s}\}\right\}
\end{align*}
be the subalgebra of \emph{regularly decaying} functions.

The holomorphic functional calculus for a sectorial operator $B$ of angle $\varphi\in(0,\pi)$ on a Hilbert space $\HH$ is defined as follows. For $\psi \in (\varphi, \pi)$ and $f\in \H_0^\infty (\S_\psi^+)$ define $f(B)\in \Lop(\HH)$ via the Cauchy integral
\begin{align*}
f(B):=\frac{1}{2\pi \i}\int_{\partial \S_\nu^+}f(z)(z-B)^{-1} \; \d z,
\end{align*}
where $\nu \in(\varphi,\psi)$ and the boundary curve $\partial \S_\nu^+$ surrounds $\sigma(B)$ counterclockwise. This integral converges absolutely and is independent of the particular choice of $\nu$ due to Cauchy's theorem. Furthermore, define $g(B):=f(B)+c(1+B)^{-1}+d$ if $g$ is of the form $g=f + c(1+\z)^{-1} + d$ for $f\in\H^{\infty}_{0} (\S_\psi^+)$ and $c,d\in \IC$. This yields an algebra homomorphism
\begin{align*}
\mathcal{E}(\S_\psi^+):=\H_0^\infty (\S_\psi^+)\oplus\langle (1+\z)^{-1}\rangle\oplus\langle\ind\rangle\rightarrow \mathcal{L}(X),\quad g\mapsto g(B), 
\end{align*}
the \emph{primary holomorphic functional calculus} for the sectorial operator $B$. It can be extended to a larger class of holomorphic functions by \emph{regularization} \cite[Sec.\ 1.2]{Haase}: If $f$ is a holomorphic function on $\S_\psi^+$ for which there exists an $e \in \mathcal{E}(\S_\psi^+)$ such that $ef \in \mathcal{E}(\S_\psi^+)$ and $e(B)$ is injective, define $f(B):= e(B)^{-1}(ef)(B)$. This yields a closed and (in general) unbounded operator on $\HH$ and the definition is independent of the particular regularizer $e$. If holomorphic functions $f,g: \S_\psi^+ \to \IC$ can be regularized, then the composition rules
\begin{align}
\label{Eq: Composition in Hinfty calculus}
f(B) + g(B) \subseteq (f+g)(B) \quad \text{and} \quad f(B)g(B) \subseteq (fg)(B)
\end{align}
hold true and $\dom(f(B)g(B)) = \dom((fg)(B)) \cap \dom(g(B))$, cf.\ \cite[Prop.\ 1.2.2]{Haase}.

In particular, for each $\alpha > 0$ and each $\eps \geq 0$ the function $(\eps + \z)^\alpha$ is regularizable by $(1+\z)^{-k}$ for $k$ a natural number larger than $\alpha$ and yields the \emph{fractional power} $(\eps + B)^\alpha$. The domain of $(\eps + B)^\alpha$ is independent of $\eps \geq 0$. Many rules for fractional powers of complex numbers remain valid for these operators, see \cite[Sec.\ 3.1]{Haase} for details. If $B$ is injective, then each $f \in \H^\infty(\S_\psi^+)$ is regularizable by $\z(1+\z)^{-2}$ yielding the \emph{$\H^\infty(\S_\psi^+)$-calculus} for $B$. 

The holomorphic functional calculus for bisectorial operators can be set up in exactly the same manner by replacing sectors $\S_\psi^+$ by the respective bisectors $\S_\psi$ and resolvents $(1+B)^{-1}$ by $(\i + B)^{-1}$. It shares all properties of the sectorial calculus listed above. 

If $B$ is bisectorial of angle $\varphi \in (0 ,\frac{\pi}{2})$, then $B^2$ is sectorial of angle $2 \varphi$. We remark that this correspondence is compatible with the respective functional calculi.

\begin{lemma}
\label{Lem: Compatibility of calculi for sectorial and bisectorial operators}
Let $B$ be a bisectorial operator of angle $\varphi \in (0, \frac{\pi}{2})$ on a Hilbert space $\HH$, let $\psi \in (\varphi, \frac{\pi}{2})$, and let $f \in \H_0^\infty(\S_{2\psi}^+)$. Then $f(\z^2)(B)$ and $f(B^2)$ defined via the holomorphic functional calculi for the bisectorial operator $B$ and the sectorial operator $B^2$ respectively, coincide. 
\end{lemma}

\begin{proof}
Note that $\z^2$ maps the bisector $\S_\psi$ onto the sector $\S_{2\psi}^+$. Hence, $g:=f(\z^2) \in \H_0^\infty(\S_\psi)$ and the claim follows by a straightforward transformation of the defining Cauchy integrals.
\end{proof}

\begin{corollary}
\label{Cor: Square root compatibilty}
Suppose the setting of Lemma~\ref{Lem: Compatibility of calculi for sectorial and bisectorial operators} and let $\beta > 0$. Then $(\z^2)^\beta(B) = \z^\beta(B^2)$.
\end{corollary}

\begin{proof}
Let $k \in \IN$ be larger than $\beta$. It suffices to remark that $e:=(1 +\z)^{-k}$ regularizes $\z^\beta$ in the functional calculus for $B^2$ and that $e(\z^2)$ regularizes $(\z^2)^\beta$ in the functional calculus for $B$.
\end{proof}

\subsection{Boundedness of the $\H^\infty$-calculus for bisectorial operators} 
\label{Subsec: Boundedness Hinfty calculus}

Given an injective bisectorial operator $B$  of angle $\varphi \in (0, \frac{\pi}{2})$ on a Hilbert space $\HH$ and some angle $\psi \in (\varphi, \frac{\pi}{2})$, the $\H^\infty(\S_\psi)$-calculus for $B$ is said to be \emph{bounded} with \emph{bound} $C_\psi > 0$ if  
\begin{align*}
 \|f(B)\|_{\Lop(\HH)} \leq C_\psi \|f\|_{\infty, \S_\psi} \qquad (f \in \H^\infty(\S_\psi)).
\end{align*}
It is convenient that boundedness of the $\H^\infty(\S_\psi)$-calculus follows from a uniform bound for the $\H_0^\infty(\S_\psi)$-calculus. Upon replacing sectors by bisectors and the regularizer $\z(1 +\z)^{-2}$ by $\z^2(1+\z^2)^{-2}$, the same argument as in \cite[Sec.\ ~5.3.4]{Haase} applies in the bisectorial case yielding

\begin{proposition}
\label{Prop: H0infty calculus bounds Hinfty calculus}
Let $B$ be an injective bisectorial operator of angle $\varphi \in (0, \frac{\pi}{2})$ on a Hilbert space $\HH$ and let $\psi \in (\varphi, \frac{\pi}{2})$. If there exists a constant $C_\psi > 0$ such that
\begin{align*}
 \|f(B)\|_{\Lop(\HH)} \leq C_\psi \|f\|_{\infty, \S_\psi} \qquad (f \in \H_0^\infty(\S_\psi)),
\end{align*}
then the $\H^\infty(\S_\psi)$-calculus for $B$ is bounded with bound $C_\psi$.
\end{proposition}

On Hilbert spaces boundedness of the $\H^\infty$-calculus is equivalent to certain \emph{quadratic estimates}, see e.g.\ \cite{CowlingDustMcIntoshYagi} for the sectorial case. Likewise, in the bisectorial case the following holds.

\begin{proposition}
\label{Prop: Quadratic estimate implies bounded Hinfty calculus}
Let $B$ be an injective bisectorial operator of angle $\varphi \in (0, \frac{\pi}{2})$ on a Hilbert space $\HH$. If $B$ satisfies quadratic estimates
\begin{align*}
 \int_0^\infty \|t B(1+ t^2 B^2)^{-1}u \|_{\HH}^2 \; \frac{\d t}{t} \simeq \|u\|_{\HH}^2 \qquad (u \in \HH),
\end{align*}
then the $\H^\infty(\S_\psi)$-calculus for $B$ is bounded for each $\psi \in (\varphi, \frac{\pi}{2})$.
\end{proposition}

For later references we include a proof drawing upon the following lemma.

\begin{lemma}[{\cite[p.\ 473]{AKM-QuadraticEstimates}}]
 \label{Lem: Resolution of the identity}
If $B$ is a bisectorial operator on a Hilbert space $\HH$, then
\begin{align*}
 \lim_{\substack{r \to 0 \\ R \to \infty}} \int_r^R (t B(1+ t^2 B^2)^{-1})^2 u \; \frac{\d t}{t} = \frac{1}{2} u \qquad (u \in \overline{\Rg(B)}).
\end{align*}
\end{lemma}

\begin{proof}[Proof of Proposition~\ref{Prop: Quadratic estimate implies bounded Hinfty calculus}]
We appeal to Proposition~\ref{Prop: H0infty calculus bounds Hinfty calculus}. Fix $\psi \in (\varphi, \frac{\pi}{2})$ and $f \in \H_0^{\infty} (\S_{\psi})$. For $t>0$ put $\Psi_t := t \z (1 + t^2 \z^2)^{-1} \in \H_0^\infty(\S_\psi)$. The most direct estimate on the defining Cauchy integral gives
\begin{align}
\label{Ineq: Auxiliary estimate for Schur's estimate}
\begin{split}
 \| \Psi_t(B) f(B) \Psi_s(B)\|_{\Lop(\HH)} 
&\lesssim \| f \|_{\infty,\S_\psi} \int_0^{\infty} \frac{t s^{-1} r}{( 1 + (ts^{-1} r)^2)( 1 + r^2 )} \; \d r =: \| f \|_{\infty,\S_\psi} \zeta(ts^{-1}) 
\end{split}
\end{align}
for all $s,t > 0$ and an implicit constant depending only on $\psi$. Here, $\zeta \in \L^1 (0, \infty; \d r / r)$. Recall $\HH = \Ke(B) \oplus \cl{\Rg(B)} = \cl{\Rg(B)}$ as $B$ is injective. For $u \in \HH$ apply the quadratic estimate to $f(B)u$ and then use Lemma~\ref{Lem: Resolution of the identity} for $u$ to find
\begin{align*}
\| f(B) u \|_{\HH}^2 
&\lesssim \int_0^{\infty} \| \Psi_t(B)f(B) u \|_{\HH}^2 \; \frac{\d t}{t}
\lesssim \int_0^{\infty} \bigg( \int_0^{\infty} \| \Psi_t(B) f(B) \Psi_s(B) \Psi_s(B) u \|_{\HH} \; \frac{\d s}{s} \bigg)^2 \; \frac{\d t}{t}. 
\intertext{By \eqref{Ineq: Auxiliary estimate for Schur's estimate} and H\"older's inequality,}
&\lesssim \| f \|_{\infty,\S_\psi}^2 \int_0^{\infty} \bigg( \int_0^{\infty} \zeta(ts^{-1}) \; \frac{\d s}{s} \bigg) \bigg( \int_0^{\infty} \zeta(ts^{-1}) \| \Psi_s(B) u \|_{\HH}^2 \; \frac{\d s}{s} \bigg) \; \frac{\d t}{t}.
\end{align*}
The right-hand side is bounded by $\| f \|_{\infty, \S_\psi}^2 \|\zeta\|_{\L^1(0,\infty; \d r / r)}^2 \| u \|_{\HH}^2$.
\end{proof}

\begin{remark}
\label{Rem: Self-adjoint operators have quadratic estimates}
Suppose that $B$ is a self-adjoint (and hence bisectorial) operator on a Hilbert space $\HH$. Then $\Psi_t(B)= t B(1+ t^2 B^2)^{-1}$ is self-adjoint for each $t>0$ and Lemma \ref{Lem: Resolution of the identity} yields
\begin{align*}
 \int_0^{\infty} \|\Psi_t(B) u\|_{\HH}^2 \; \frac{\d t}{t} = \lim_{\substack{r \to 0 \\ R \to \infty}} \Big\langle \int_r^R \Psi_t(B)^2 u \; \frac{\d t}{t} , u \Big\rangle_{\HH} = \frac{1}{2} \|u\|_{\HH}^2 \qquad(u \in \cl{\Rg(B)}).
\end{align*}
The proof of Proposition~\ref{Prop: Quadratic estimate implies bounded Hinfty calculus} then reveals the following: If $\{T_t\}_{t >0} \subseteq \Lop(\HH)$ is a family of operators for which there is $\zeta \in \L^1(0, \infty; \d r/r)$ such that $\| T_t \Psi_s(B)\|_{\Lop(\HH)} \lesssim \zeta(t s^{-1})$ for all $s,t > 0$, then
\begin{align*}
 \int_0^\infty \|T_t u\|_{\HH}^2 \; \frac{\d t}{t} \lesssim \|u\|_{\HH}^2 \qquad (u \in \cl{\Rg(B)}).
\end{align*}
This is usually called a \emph{Schur type estimate}. In the proof of Proposition~\ref{Prop: Quadratic estimate implies bounded Hinfty calculus}, $T_t = \Psi_t(B) f(B)$.
\end{remark}

For completeness we add a short proof of Corollary~\ref{Cor: Abstract Kato problem}.

\begin{proof}[Proof of Corollary~\ref{Cor: Abstract Kato problem}]
The first part of the corollary is due to $\HH = \Ke(\pb) \oplus \cl{\Rg(\pb)}$ and Proposition~\ref{Prop: Quadratic estimate implies bounded Hinfty calculus}. Put $T:= \pb|_{\cl{\Rg(\pb)}}$. As $\frac{\z}{\sqrt{\z^2}}, \frac{\sqrt{\z^2}}{\z} \in \H^\infty(\S_\psi)$, the composition rules \eqref{Eq: Composition in Hinfty calculus} yield
\begin{align*}
 \dom({\textstyle \sqrt{\pb^2}}) \cap \cl{\Rg(\pb)} =  \dom(\pb) \cap \cl{\Rg(\pb)} \quad \text{with} \quad \|{\textstyle \sqrt{\pb^2}} u\|_\HH \simeq \|\pb u\|_\HH \qquad (u \in \dom(\pb) \cap \cl{\Rg(\pb)}).
\end{align*}
Here, we used $\dom(T) = \dom(\pb) \cap \cl{\Rg(\pb)}$ and $\dom(\sqrt{T^2}) = \dom(\sqrt{\pb^2}) \cap \cl{\Rg(\pb)}$, the latter following as in \cite[Prop.\ 2.6.5]{Haase}. The Kato square root type estimate follows from $\Ke(\pb) \subseteq \Ke(\sqrt{\pb^2})$, which can be deduced exactly as in \cite[Thm.\ 2.3.3c]{Haase}. 
\end{proof}

\section{The hypotheses underlying Theorem~\ref{Thm: The pibe theorem}}
\label{Sec: The abstract framework}

In this section we introduce the hypotheses \eqref{H1} - \eqref{H7} underlying Theorem~\ref{Thm: The pibe theorem} and summarize their well-established operator theoretic consequences. The first four of our hypotheses are:

\begin{enumerate}

\labitem{H1}{H1}{The operator $\g$ is \emph{nilpotent}, i.e.\ closed, densely defined, and satisfies $\Rg(\g) \subseteq \Ke(\g)$. In particular $\g^2 = 0$ on $\dom(\g)$.}

\labitem{H2}{H2}
{The operators $B_1$ and $B_2$ are defined everywhere on $\HH$. There exist $\kappa_1 , \kappa_2 > 0$ such that they satisfy the \emph{accretivity conditions}
\begin{align*}
\Re \langle B_1 u , u \rangle_\HH &\geq \kappa_1 \| u \|_{\HH}^2 \qquad (u \in \Rg(\g^*)), \\
\Re \langle B_2 u , u \rangle_\HH &\geq \kappa_2 \| u \|_{\HH}^2 \qquad (u \in \Rg(\g)).
\end{align*}
and there exist $K_1, K_2$ such that they satisfy the \emph{boundedness conditions}
\begin{align*}
 \|B_1 u\|_\HH \leq K_1 \|u\|_\HH \quad \text{and} \quad \|B_2 u\|_\HH \leq K_2 \|u\|_\HH \qquad (u \in \HH).
\end{align*}}

\labitem{H3}{H3}{The operator $B_2 B_1$ maps $\Rg(\g^*)$ into $\Ke(\g^*)$ and the operator $B_1 B_2$ maps $\Rg(\g)$ into $\Ke(\g)$. In particular, $\g^* B_2 B_1 \g^* = 0$ on $\dom(\g^*)$ and $\g B_1 B_2 \g = 0$ on $\dom(\g)$.}

\labitem{H4}{H4}{The operators $B_1, B_2$ are multiplication operators induced by  $\L^{\infty}(\Omega ; \Lop(\IC^N))$-functions.}
\end{enumerate}

We define the \emph{Dirac type operator} $\Pi:= \g + \g^*$ and the perturbed operators $\gbs := B_1 \g^* B_2$ and $\pb := \g + \gbs$. The first three hypotheses trace out the classical setup for perturbed Dirac type operators introduced in \cite{AKM}. They have the following consequences. Firstly, \eqref{H1} implies that $\g^*$ is nilpotent and so is $\gb^*$, cf.\ \cite[Lem.\ 4.1]{AKM-QuadraticEstimates}. The operator $\pb$ induces the algebraic and topological \emph{Hodge decomposition}
\begin{align}
\label{Eq: Hodge Decomposition}
\HH = \Ke(\pb) \oplus \overline{\Rg(\gbs)} \oplus \overline{\Rg(\g)}
\end{align}
and in particular
\begin{align}
\label{Eq: Kernel and range decompositions}
 \Ke(\pb) = \Ke(\gbs) \cap \Ke(\g) \quad \text{and} \quad \overline{\Rg(\pb)} = \overline{\Rg(\gbs)} \oplus \overline{\Rg(\g)}
\end{align}
hold \cite[Prop.\ 2.2]{AKM-QuadraticEstimates}. Moreover, $\pb$ is bisectorial of some angle $\omega \in (0, \frac{\pi}{2})$, cf.\ \cite[Prop. 2.5]{AKM-QuadraticEstimates}. Consequently, $\pb^2$ is sectorial of angle $2 \omega$. The unperturbed operator $\Pi$ is self-adjoint \cite[Cor.\ ~4.3]{AKM-QuadraticEstimates} and thus satisfies quadratic estimates, cf.\ Remark~\ref{Rem: Self-adjoint operators have quadratic estimates}. In particular, $\dom(\sqrt{\Pi^2}) = \dom(\Pi)$ with equivalence of the homogeneous graph norms as in Corollary~\ref{Cor: Abstract Kato problem}. Finally, if $\g$ satisfies \eqref{H1}, then \eqref{H2} and \eqref{H3} are always satisfied for $B_1 = B_2 = \Id$ and hence the results above remain true in the \emph{unperturbed setting} when $\gb^* = \g^*$ and $\pb = \Pi$. 

\begin{remark}
\label{Rem: Uniformity of constants}
In all results from \cite{AKM-QuadraticEstimates} implicit constants depend on the perturbations $B_1$ and $B_2$ only through the constants $\kappa_{1,2}, K_{1,2}$ quantified in \eqref{H2}. This has already been stated in \cite[Sec.~2]{AKM-QuadraticEstimates} and has been worked out in greatest details in the master's thesis of one of the authors \cite{Patrick-Masterarbeit}.
\end{remark}

Similar to \cite{AKM-QuadraticEstimates, AKM} the set of hypotheses is completed by localization and coercivity assumptions on the unperturbed operators. The slight difference between \eqref{H7} and the corresponding hypothesis in \cite{AKM} stresses that no further knowledge on the occurring interpolation spaces is necessary.

\begin{enumerate}

 \labitem{H5}{H5}{For every $\varphi \in \C_c^{\infty} (\IR^d ; \IC)$ the associated multiplication operator $M_\varphi$ maps $\dom(\g)$ into itself and the commutator $[\g , M_{\varphi} ] = \g M_{\varphi} - M_{\varphi} \g$ with domain $\dom([\g , M_{\varphi}]) = \dom(\g)$ acts as a multiplication operator induced by some $c_\varphi \in \L^\infty(\Omega; \Lop(\IC^N))$ with entries
\begin{align*}
\lvert c_{\varphi}^{i,j} (x) \rvert \lesssim \lvert \nabla \varphi(x) \rvert \qquad (x \in \Omega, \, 1 \leq i, j \leq N)
\end{align*} for an implicit constant independent of $\varphi$.}

 \labitem{H6}{H6} {For every open ball $B$ centered in $\Omega$, and for all $u \in \dom(\g)$ and $v \in \dom(\g^*)$ with compact support in $B \cap \Omega$ it holds
\begin{align*}
\bigg\lvert \int_{\Omega} \g u \; \d x  \bigg\rvert \lesssim \lvert B \rvert^{\frac{1}{2}} \| u \|_\HH \qquad \text{and} \qquad
\bigg\lvert \int_{\Omega} \g^* v \; \d x \bigg\rvert \lesssim \lvert B \rvert^{\frac{1}{2}} \| v \|_\HH.
\end{align*}}

 \labitem{H7}{H7}{There exist $\beta_1 , \beta_2 \in (0 , 1]$ such that the fractional powers of $\Pi^2$ satisfy
\begin{align*}
\| u \|_{[\HH , \V^k]_{\beta_1}} \lesssim \| ( \Pi^2 )^{\beta_1 / 2} u \|_\HH \qquad \text{and} \qquad \| v \|_{[\HH , \V^k]_{\beta_2}} \lesssim \| ( \Pi^2 )^{\beta_2 / 2} v \|_\HH
\end{align*}
for all $u \in \Rg(\g^*) \cap \dom(\Pi^2)$ and all $v \in \Rg(\g) \cap \dom(\Pi^2)$.
}
\end{enumerate}

\begin{remark}
\label{Rem: Hypotheses remain valid under permutation}
It is straightforward to check that if the triple of operators $\{\g , B_1 , B_2\}$ satisfies \eqref{H1} - \eqref{H7}, then so do the triples $\{\g^* , B_2 , B_1\} , \{\g^* , B_2^* , B_1^*\}$, and $\{\g , B_1^* , B_2^*\}$.
\end{remark}

\section{The proof of Theorem~\ref{Thm: Main Theorem}}
\label{Sec: Proof of main result}

\noindent In this section we deduce Theorem~\ref{Thm: Main Theorem} from Theorem~\ref{Thm: The pibe theorem} applied on $\HH := \L^2 (\Omega ; \IC^m) \times \L^2(\Omega; \IC^m) \times (\L^2 (\Omega ; \IC^m))^d$. The argument is similar to \cite{AKM}. 

Recall that $\a: \V \times \V \to \IC$ is the sesquilinear form corresponding to $Au = -\sum_{\alpha, \beta = 1}^d \partial_\alpha(a_{\alpha, \beta} \partial_{\beta} u)$ and let $\mathfrak{A}$ be the multiplication operator corresponding to the coefficient tensor $(a_{\alpha,\beta})_{1\leq \alpha,\beta \leq d} \in \L^\infty(\Omega; \Lop(\IC^{dm}))$. Define $\nabla_\V u:= \nabla u$ on $\dom(\nabla_\V):= \V$ and put
\begin{align*}
 \g:= \begin{bmatrix} 0 & 0 & 0 \\ \Id & 0 & 0 \\ \nabla_\V & 0 & 0 \end{bmatrix}, \quad
 B_1:= \begin{bmatrix} \Id & 0 & 0 \\ 0 & 0 &0 \\ 0 & 0 & 0\end{bmatrix}, \quad \text{and} \quad 
 B_2:= \begin{bmatrix} 0 & 0 & 0 \\ 0 & \Id & 0 \\ 0 & 0 & \mathfrak{A}  \end{bmatrix}
\end{align*}
on their natural domains. By these choices 
\begin{align*}
 \pb =  \begin{bmatrix} 0 & \Id & (\nabla_\V)^* \mathfrak{A} \\ \Id & 0 & 0 \\ \nabla_\V & 0 & 0 \end{bmatrix} \quad \text{and} \quad 
 \pb^2 = \begin{bmatrix} 1 + A & 0 & 0 \\ 0 & \Id & (\nabla_\V)^* \mathfrak{A} \\ 0 & \nabla_\V & \nabla_\V (\nabla_\V)^* \mathfrak{A} \end{bmatrix}.
\end{align*}
The corresponding unperturbed operators $\Pi$ and $\Pi^2$ are obtained by replacing $\mathfrak{A}$ by $\Id$ and $A$ by $-\Delta_\V$. Upon restricting to the first component of $\HH$, these representations show that Theorem~\ref{Thm: Main Theorem} follows from $\dom(\sqrt{\pb^2}) = \dom(\pb)$ with equivalences of the homogeneous graph norms, cf.\ Corollary~\ref{Cor: Abstract Kato problem}. So, to complete the proof of Theorem~\ref{Thm: Main Theorem} it remains to verify \eqref{H1} - \eqref{H7} for these particular choices of operators.

\subsection{Verification of \eqref{H1} - \eqref{H7}}
It is obvious that \eqref{H1}, \eqref{H3}, and \eqref{H4} hold. Also \eqref{H2} is immediate for $B_1$ and for $B_2$ it follows from Assumption~\ref{Ass: Ellipticity}. The validity of \eqref{H5} is a consequence of \eqref{V} in Assumption~\ref{Ass: Geometric setup} and the product rule. 

Since the integral over the gradient of a compactly supported function vanishes, the estimate for $u$ in \eqref{H6} follows from H\"older's inequality. For $v$ take $\varphi \in \C_c^{\infty} (\Omega ; \IR)$ with $\varphi \equiv 1$ on $\supp(v)$ and denote by $\{e_j\}_{j = 1}^{(d+2)m}$ the standard basis of $\IC^{(d+2)m}$. Note $\supp(\g^* v) \subseteq \supp(v)$ by \eqref{H5} for $\g^*$ in place of $\g$, cf.\ Remark~\ref{Rem: Hypotheses remain valid under permutation}. As
\begin{align*}
 \varphi e_j \in \H_0^1(\Omega; \IC^m)^{d+2} \subseteq \V^{d+2} \subseteq \dom(\g)
\end{align*}
for each $j$ by Assumption~\ref{Ass: Geometric setup}, it follows
 \begin{align*}
 \bigg \lvert \int_{\Omega} \g^*v \; \d x \bigg\rvert 
 \simeq \sum_{j = 1}^{(d+2)m} \bigg\lvert \int_{\Omega} \langle \varphi e_j , \g^* v \rangle \; \d x \bigg\rvert
 = \sum_{j = 1}^{(d+2)m} \bigg\lvert \int_{\Omega} \langle \g (\varphi  e_j) , v \rangle \; \d x\bigg\rvert.
 \end{align*}
Since $\abs{\g (\varphi  e_j)} \leq 1$ a.e.\ on $\supp(v)$, the required estimate is obtained by H\"older's inequality. 

For the first part of \eqref{H7} take $\beta_1 = 1$ and note
\begin{align*}
 \|u\|_{\V^{d+2}} =  \|u\|_{\H^1(\Omega; \IC^m)^{d+2}} = \|\g u\|_{\HH} = \|\Pi u\|_{\HH} \simeq \|{\textstyle \sqrt{\Pi^2}} u \|_{\HH}.
\end{align*}
For the second part take $\beta_2 = \alpha$ as in Assumption~\ref{Ass: Geometric setup}. Fix $v = (0, w, \nabla_\V w) \in \Rg(\g) \cap \dom(\Pi^2)$. Then $w \in \dom((\nabla_\V)^* \nabla_\V) = \dom(1 - \Delta_\V)$ so that by Assumption \eqref{Emb: Embedding condition} of Theorem~\ref{Thm: Main Theorem},
\begin{align*}
 \|v\|_{[\HH , \V^{d + 2}]_{\alpha}} 
&\simeq \|w\|_{\H^{\alpha,2}(\Omega; \IC^m)} +  \|\nabla w\|_{\H^{\alpha,2}(\Omega; \IC^m)^d} \\
&\lesssim \|w\|_{\H^{1+\alpha,2}(\Omega; \IC^m)} \\
&\lesssim \|(1 - \Delta_\V)^{1 / 2 + \alpha / 2} w\|_{\L^2(\Omega ; \IC^m)}.
\end{align*}
However, $(1 - \Delta_\V)^{1 / 2 + \alpha / 2} w = (\Pi^2)^{1/2 + \alpha/2}\widehat{w}$, where $\widehat{w} = (w,0,0) \in \dom(\Pi)$. Thus, Corollary~\ref{Cor: Square root compatibilty} and the composition rules \eqref{Eq: Composition in Hinfty calculus} for the functional calculus for $\Pi$ yield
\begin{align*}
\|(\Pi^2)^{1 / 2}(\Pi^2)^{\alpha / 2} \widehat{w}\|_{\HH} \simeq \| \Pi (\Pi^2)^{\alpha / 2} \widehat{w}\|_{\HH} = \| (\Pi^2)^{\alpha / 2} \Pi \widehat{w}\|_{\HH} = \|(\Pi^2)^{\alpha/2} v\|_{\HH}
\end{align*}
as required.

\section{The Proof of Theorem~\ref{Thm: The pibe theorem}: Preliminaries}
\label{Sec: Reduction to finite time}

\noindent In this and the following two sections we develop the proof of Theorem~\ref{Thm: The pibe theorem}. Throughout we assume that $\g$, $B_1$, and $B_2$ are operators on $\HH$ satisfying \eqref{H1} - \eqref{H7}. We shall stick to the notions introduced in Section~\ref{Sec: The abstract framework} but simply write $\| \cdot \|$ instead of $\| \cdot \|_{\HH}$ as long as no misunderstandings are expected. We shall use the discussed properties of $\g$, $\g^*$, $\Pi$, $\gb^*$, and $\pb^*$ without further referencing. We also introduce the following bounded operators on $\HH$:
\begin{align*}
 R_t^B &:= ( 1 + \i t \pb )^{-1}, \quad P_t^B := ( 1 + t^2 \pb^2 )^{-1}, \quad Q_t^B = t\pb P_t^B, \quad \text{and} \quad \Theta_t^B:= t \gbs P_t^B \qquad (t \in \IR).
\end{align*}
In the unperturbed case, i.e.\ if $B_1 = B_2 = \Id$, we simply write $R_t, P_t$, $Q_t$, and $\Theta_t$.

In order to carry out correctly the dependence of the implicit constants on the perturbations $B_1$ and $B_2$, we make the following
\begin{agreement}
\label{Agr: Implicit constants}
In the proof of Theorem~\ref{Thm: The pibe theorem} the symbols $\lesssim$, $\gtrsim$, and $\simeq$ are reserved for estimates invoking implicit constants whose dependence on $B_1$ and $B_2$ is only through the constants quantified in \eqref{H2}.
\end{agreement}

\begin{lemma}
\label{Lem: Basic properties of resolvent operators}
For each $t \in \IR$ it holds $P_t^B = \frac{1}{2} ( R_t^B + R_{-t}^B ) = R_t^B R_{-t}^B$ and $Q_t^B = \frac{1}{2\i} ( R_{-t}^B - R_t^B )$. Moreover,
\begin{align*}
 \|R_t^B\|_{\Lop(\HH)} + \|P_t^B\|_{\Lop(\HH)} + \|Q_t^B\|_{\Lop(\HH)} + \|\Theta_t^B\|_{\Lop(\HH)} \lesssim 1 \qquad (t \in \IR).
\end{align*}
\end{lemma}

\begin{proof}
Checking the identities is a straightforward calculation. The boundedness of $\{R_t^B\}_{t \in \IR}$, $\{P_t^B\}_{t \in \IR}$, and $\{Q_t^B\}_{t \in \IR}$ then follows by bisectoriality of $\pb$. Finally, $\|\Theta_t^B\|_{\Lop(\HH)} \lesssim \|Q_t^B\|_{\Lop(\HH)}$ holds for all $t \in \IR$ due to the topological decomposition $\cl{\Rg(\pb)} = \cl{\Rg(\gbs)} \oplus \cl{\Rg(\g)}$, cf.\ ~\eqref{Eq: Kernel and range decompositions}.
\end{proof}

In \cite[Prop.\ 4.8]{AKM-QuadraticEstimates} Axelsson, Keith, and M\textsuperscript{c}Intosh reveal that \eqref{H1}-\eqref{H3} already imply
\begin{align}
\label{Eq: Abstract setting quadratic estimate}
 \int_0^\infty \|\Theta_t^B (1-P_t)u\|^2 \; \frac{\d t}{t} \lesssim \|u\|^2 \qquad (u \in \Rg(\g)),
\end{align}
and that a sufficient condition for the quadratic estimate \eqref{Eq: Quadratic estimate for pb} for $\pb$ is
\begin{align}
\label{Eq: The quadratic estimate}
\int_0^{\infty} \| \Theta_t^B P_t u \|^2 \; \frac{\d t}{t} \lesssim \|u\|^2 \qquad (u \in \Rg(\g))
\end{align}
and the three analogous estimates obtained by replacing $\{\g , B_1 , B_2\}$ by $\{\g^* , B_2 , B_1\} , \{\g^* , B_2^* , B_1^*\}$, and $\{\g , B_1^* , B_2^*\}$. In fact, owing to Remark~\ref{Rem: Hypotheses remain valid under permutation}, it suffices to prove \eqref{Eq: The quadratic estimate}. In this section we shall take care of the integral over $t \geq 1$ and decompose the remaining finite time integral into three pieces that will be handled later on.

\begin{lemma}[Reduction to finite time]
\label{Lem: Reduction to finite time}
It holds
\begin{align*}
\int_1^{\infty} \| \Theta_t^B P_t u \|^2 \; \frac{\d t}{t} \lesssim \|u\|^2 \qquad (u \in \Rg(\g)).
\end{align*}
\end{lemma}

\begin{proof}
Fix $u = \g w \in \Rg(\g)$. By nilpotence of $\g$ and $\g^*$ one readily checks
\begin{align}
\label{Eq: Pt commutes with g}
 P_t u = (1+ t^2 \Pi^2)^{-1} \g (1+ t^2 \Pi^2)(1+ t^2 \Pi^2)^{-1} w = \g (1+ t^2 \Pi^2)^{-1} w = \g P_t w \quad (t \in \IR \setminus \{0\}).
\end{align}
Hence, the second part of \eqref{H7} applies to $v = P_t u$. Lemma~\ref{Lem: Basic properties of resolvent operators} and the continuous inclusion $[\HH, \V^{k}]_{\beta_2} \subseteq \HH$, yield
\begin{align*}
 \int_1^{\infty} \|\Theta_t^B P_t u\|^2 \; \frac{\d t}{t} 
&\lesssim \int_1^{\infty} \| P_t u\|_{[\HH, \V^k]_{\beta_2}}^2 \; \frac{\d t}{t}  
\lesssim \int_1^{\infty} \|t^{\beta_2} (\Pi^2)^{\beta_2 / 2} P_t u\|^2 \; \frac{\d t}{t^{1 + 2 \beta_2}}.
\intertext{Define regularly decaying holomorphic functions $f_t := (t^2 \z)^{\beta_2 / 2} (1 + t^2 \z)^{-1}$. A direct estimate on the defining Cauchy integral yields a bound for $\|f_t(\Pi^2)\|_{\Lop(\HH)}$ uniformly in $t \geq 1$. Thus,}
& = \int_1^{\infty} \|f_t(\Pi^2)u \|^2 \; \frac{\d t}{t^{1 + 2 \beta_2}} \lesssim \int_1^{\infty} \|u\|^2 \; \frac{\d t}{t^{1 + 2 \beta_2}} = \frac{1}{2 \beta_2} \|u\|^2. \qedhere
\end{align*}
\end{proof}

To proceed further, we introduce a slightly modified version of Christ's dyadic decomposition for doubling metric measure spaces \cite[Thm.\ 11]{ChristDyadic}. In fact, if one aims only at a \emph{truncated dyadic cube structure} with a common bound for the diameter of all dyadic cubes, then Christ's argument literally applies to locally doubling metric measure spaces. This has been previously noticed e.g.\ by Morris \cite{Morris}. Here, a metric measure space $X$ with metric $\rho$ and positive Borel measure $\mu$ is \emph{doubling} if there is a constant $C> 0$ such that
\begin{align*}
 \mu (\{x \in X: \rho(x,x_0) < 2r \}) \leq C \mu (\{x \in X: \rho(x,x_0) < r \}) \qquad (x_0 \in X, \, r>0)
\end{align*}
and it is \emph{locally doubling} if the above holds for all $x_0 \in X$ and all $r \in (0,1]$. Note that \eqref{d} of Assumption~\ref{Ass: Geometric setup} entails that $\Omega$ equipped with the restricted Euclidean metric and the restricted Lebesgue measure is locally doubling.

\begin{theorem}[Christ]
 \label{Thm: Dyadic decomposition}
Under Assumption~\ref{Ass: Geometric setup}.\eqref{d} there exists a collection of open subsets \linebreak[4] $\{Q_{\alpha}^k \subseteq \Omega : k \in \IN_0 , \alpha \in I_k \}$, where $I_k$ are index sets, and constants $\delta \in (0 , 1)$ and $a_0 , \widehat{\eta} , C_1 , \widehat{C}_2 > 0$ such that:
\begin{enumerate}

\labitem{i}{i}{$\lvert \Omega \setminus \bigcup_{\alpha \in I_k} Q_{\alpha}^k \rvert = 0$ for each $k \in \IN_0$.}

\labitem{ii}{ii}{If $l \geq k$ then for each $\alpha \in I_k$ and each $\beta \in I_l$ either $Q_{\beta}^l \subseteq Q_{\alpha}^k$ or $Q_{\beta}^l \cap Q_{\alpha}^k = \emptyset$ holds.}

\labitem{iii}{iii}{If $l \leq k$ then for each $\alpha \in I_k$ there is a unique $\beta \in I_l$ such that $Q_{\alpha}^k \subseteq Q_{\beta}^l$.}

\labitem{iv}{iv}{It holds $\diam(Q_{\alpha}^k) \leq C_1 \delta^k$ for each $k \in \IN_0$ and each $\alpha \in I_k$.}

\labitem{v}{v}{For each $Q_\alpha^k$, $k \in \IN_0$, $\alpha \in I_k$, there exists $z_{\alpha}^k \in \Omega$ such that $B(z_{\alpha}^k , a_0 \delta^k) \cap \Omega \subseteq Q_{\alpha}^k$.}

\labitem{vi}{vi}{If $k \in \IN_0$, $\alpha \in I_k$, and $t > 0$ then $\abs{\{ x \in Q_{\alpha}^k : \dist (x , \Omega \setminus Q_{\alpha}^k) \leq t \delta^k \}} \leq \widehat{C}_2 t^{\widehat{\eta}} \abs{Q_{\alpha}^k}$.}
\end{enumerate}
\end{theorem}

By a slight abuse of notation we refer to the $Q^k_{\alpha}$ as \emph{dyadic cubes}. We denote the family of all dyadic cubes by $\Delta$ and each family of fixed \emph{step size} $\delta^k$ by $\Delta_{\delta^k} := \{Q^k_{\alpha} : \alpha \in I_k\}$. Moreover, if $k \in \IN_0$ and $t \in ( \delta^{k + 1} , \delta^{k}]$, then the family of dyadic cubes of step size $t$ is $\Delta_t := \Delta_{\delta^k}$. The \emph{sidelength} of $Q \in \Delta_{\delta^k}$ is $l(Q):= \delta^k$.

\begin{remark}
\label{Rem: Remark to dyadic decomposition}
\begin{enumerate}

 \item Assumption \ref{Ass: Geometric setup}.\eqref{d} in combination with \eqref{iv} and \eqref{v} of Theorem \ref{Thm: Dyadic decomposition} imply $\abs{Q} \simeq l(Q)^d$ for all $Q \in \Delta$.

 \item Since the dyadic cubes are open, for each $t \in (0 , 1]$ the family $\Delta_t$ is countable.

 \item The first item of Theorem \ref{Thm: Dyadic decomposition} implies that there exists a nullset $\NN \subseteq \Omega$ such that for each $t \in (0 , 1]$ and each $x \in \Omega \setminus \NN$ there exists a unique cube $Q \in \Delta_t$ that contains $x$.
\end{enumerate}
\end{remark}

A substantial drawback of Theorem~\ref{Thm: Dyadic decomposition} is that part \eqref{vi} gives an estimate for the inner boundary strips of dyadic cubes only near their relative boundary with respect to $\Omega$. This of course is a relict of the very construction. The Ahlfors-David condition is an appropriate measure-theoretic assumption on $\bd \Omega$ allowing to control the measure of the complete inner boundary strip.

Some variant of the following lemma may be well known but for the reader's convenience we include a proof.

\begin{lemma}
\label{Lem: Boundary strip lemma}
If $\Xi \subseteq \IR^d$ is open and $\bd \Xi$ is a $(d-1)$-set, then for each $r_0, t_0 > 0$ there exists $C>0$ such that
\begin{align*}
 \abs{\{x \in \Xi: \abs{x-x_0} < r, \, \dist(x, \IR^d \setminus \Xi) \leq tr \}} \leq C t r^d 
\end{align*}
for all $x_0 \in \cl{\Xi}$, $r \in (0, r_0]$, and $t \in (0,t_0]$.
\end{lemma}

\begin{proof}
For $x_0 \in \cl{\Xi}$, $r \in (0, r_0]$, and $t \in (0,t_0]$ put $E:= \{x \in \Xi: \abs{x-x_0} < r, \, \dist(x, \IR^d \setminus \Xi) \leq tr \}$. Then for each $x \in E$ there exists a boundary point $b_x \in \bd \Xi$ such that $x \in \cl{B(b_x, tr)}$. The Vitali covering lemma \cite[Sec.\ 1.5]{Fine-Properties-Of-Functions} yields a countable subset $J \subseteq E$ such that the balls $\{B(b_x, tr)\}_{x \in J}$ are pairwise disjoint and such that $\{B(b_x, 6tr)\}_{x \in J}$ is a covering of $E$. Hence, $\abs{E} \lesssim \#J (tr)^d$, where $\# J$ denotes the number of elements contained in $J$. 

To get control on $\# J$ fix $z \in J$. If $y \in B(b_x, tr)$ for some $x \in J$ then by the triangle inequality $\abs{y-b_z} \leq 3tr + 2r < (3t_0 + 2)r$. The Ahlfors-David condition $\m_{d-1}(\bd \Xi \cap B(b_x, r)) \simeq r^{d-1}$ remains valid for all $b_x \in \bd \Xi$ and all $r \in (0,(3t_0 + 2)r_0]$ with implicit constants depending only on $\Xi$, $r_0$, and $t_0$. Hence,
\begin{align*}
 ((3t_0 +2)r)^{d-1}
\gtrsim \m_{d-1}\big(\bd \Xi \cap B(b_z, (3t_0 +2)r)\big)
\geq \sum_{x \in J} \m_{d-1} \big(\bd \Xi \cap B(b_x, tr)\big). 
\end{align*}
Again by the Ahlfors-David condition the right-hand side is comparable to $\# J (tr)^{d-1}$. Thus, $\# J \lesssim ~t^{1-d}$ and the conclusion follows.
\end{proof}

As a corollary we record a connection between Ahlfors regular and plump sets that is of independent interest. Following \cite{Vaeisaelae_UniformDomains} a  bounded set $\Xi \subseteq \IR^d$ is \emph{$\kappa$-plump} if there exists $\kappa > 0$ such that for each $x_0 \in \cl{\Xi}$ and each $r \in (0, \diam(\Xi)]$ there exists $x \in \Xi$ such that $B(x, \kappa r) \subseteq \Xi \cap B(x_0,r)$.

\begin{corollary}
\label{Cor: Ahlfors implies plumb}
If $\Xi \subseteq \IR^d$ is a bounded open $d$-set and $\bd \Xi$ is a $(d-1)$-set, then $\Xi$ is $\kappa$-plump.
\end{corollary}

\begin{proof}
By the $d$-set property of $\Xi$ fix $c>0$ such that $\abs{\Xi \cap B(x_0,r)} \geq c r^d$ for all $x_0 \in \cl{\Xi}$ and all $r \in (0, \diam(\Xi)]$. Choose $r_0 := \frac{1}{2} \diam(\Xi)$ and $t_0= 1$ in Lemma~\ref{Lem: Boundary strip lemma} and apply the estimate with $t = \min\{\frac{c}{2C}, 1\}$ to conclude 
\begin{align*}
 \Big| \Big\{x \in \Xi: \abs{x-x_0} < \frac{r}{2}, \, \dist(x, \IR^d \setminus \Xi) > \frac{tr}{2} \Big \} \Big| \geq \frac{cr^d}{2^{d+1}} \qquad (x_0 \in \cl{\Xi},\, r \in (0, \diam(\Xi)]).
\end{align*}
In particular, these sets are non-empty so one can choose $\kappa = t$. 
\end{proof}

\begin{corollary}
\label{Cor: Estimate for complete boundary strip}
Under Assumptions~\ref{Ass: Geometric setup}.\eqref{d} and \ref{Ass: Geometric setup}.\eqref{d-1} there exist constants $\eta, C_2 > 0$ such that 
\begin{align*}
 \abs{\{ x \in Q : \dist (x , \IR^d \setminus Q) \leq t \delta^k \}} \leq C_2 t^{\eta} \abs{Q}
\end{align*}
for each $k \in \IN_0$, $Q \in \Delta_{\delta^k}$, and $t>0$.
\end{corollary}

\begin{proof}
Put $\eta:= \min\{1,\widehat{\eta}\}$ where $\widehat{\eta}$ is given by Theorem~\ref{Thm: Dyadic decomposition}. If $t \geq 1$ then the estimate in question holds with $C_2 = 1$. If $t < 1$ split
\begin{align*}
E:= \big \{ x \in Q : \dist (x , \IR^d \setminus Q) \leq t \delta^k \big \} \subseteq \big\{ x \in Q : \dist (x , \Omega \setminus Q) \leq t \delta^k \big\} \cup \big\{ x \in Q : \dist (x , \IR^d \setminus \Omega) \leq t \delta^k \big \}.
 \end{align*}
Property \eqref{vi} of the dyadic decomposition and Lemma~\ref{Lem: Boundary strip lemma} applied with $r_0 := C_1$, $t_0 := \frac{1}{C_1}$, and $r$ and $t$ replaced by $C_1 \delta^k$ and $\frac{t}{C_1}$ yield the estimate $\abs{E} \lesssim \widehat{C}_2 t^{\widehat{\eta}} \abs{Q} + t \delta^{kd}$. The conclusion follows from Remark~\ref{Rem: Remark to dyadic decomposition} taking into account $t<1$.
\end{proof}

The boundedness assertions of Lemma~\ref{Lem: Basic properties of resolvent operators} self-improve to off-diagonal estimates. These will be a crucial instrument in the following. Recall that given $z \in \IC$ we write $\langle z \rangle = 1 + \abs{z}$.

\begin{proposition}[Off-diagonal estimates]
\label{Prop: Off-diagonal estimates}
Let $U_t$ be either of the operators $R_t^B$, $P_t^B$, $Q_t^B$ or $\Theta_t^B$. Then for every $M \in \IN_0$ there exists a constant $A_M > 0$ such that
\begin{align*}
\big\| \ind_E U_t (\ind_F u) \big\| \lesssim A_M \Big\langle \frac{\dist(E , F)}{t} \Big\rangle^{-M} \|\ind_F u\|
\end{align*}
holds for all $u \in \HH$, all $t \in \IR \setminus \{0\}$, and all bounded Borel sets $E , F \subseteq \Omega$.
\end{proposition}

We skip the proof as it is literally the same as in \cite[Prop.\ 5.2]{AKM-QuadraticEstimates} with one minor modification: In the case $0 < \abs{t} \leq \dist(E , F)$ one separates $E$ and $F$ by some $\eta \in \C_c^{\infty}(\widetilde{E})$ such that $\eta = 1$ on $E$ and $\|\nabla \eta\|_{\infty} \leq c / \dist(E , F)$, where $\widetilde{E} := \{ x \in \IR^d : \dist(x , E) < \frac{1}{2} \dist(E , F)\}$ and $c$ depending only on $d$, rather then the choices for $\eta$ and $\widetilde{E}$ in \cite{AKM-QuadraticEstimates}. This is due to the slight difference between our \eqref{H5} and \eqref{H6} in \cite{AKM-QuadraticEstimates}.

The next lemma helps to control the sums that naturally crop up when combining off-diagonal estimates with the dyadic decomposition.

\begin{lemma}
 \label{Lem: Sums related to off-diagonal estimates}
The following hold true for each $M > d+1$.
\begin{enumerate}
\item There exists $c_M > 0$ depending solely on $M$ and $\Omega$ such that
\begin{align*}
\sum_{R \in \Delta_t} \Big\langle \frac{\dist(x , R)}{t} \Big\rangle^{-M} \leq c_M \qquad (x \in \IR^d,\, t \in (0,1]).
\end{align*}
\item Let $l \in \IN_0$, $t \in (0,1]$, $Q \in \Delta_t$, and $F \subseteq \IR^d$ be such that $\dist(Q, F) \geq l t$. Then exist $c_{l,1}, c_{l,2} \geq 0$ depending solely on $l$, $M$, and $\Omega$ such that
\begin{align*}
 \sum_{R \in \Delta_t} \Big\langle \frac{\dist(Q , R \cap F )}{s} \Big\rangle^{-M} \leq c_{l,1} + c_{l,2} \Big(\frac{s}{t}\Big)^M \qquad (s>0).
\end{align*}
If $l > 0$, then one can choose $c_{l,1} = 0$.
\end{enumerate}
\end{lemma}

\begin{proof}
To show the first statement fix $x \in \IR^d$ and $t \in (0,1]$. Fix $k \in \IN_0$ such that $\delta^{k+1} < t \leq \delta^k$. Put $\Omega_n := B(x , (n + 1) C_1 \delta^k) \cap \Omega$ for $n \in \IN_0$ and $\Omega_{-1} := \Omega_{-2} := \emptyset$. If $R \in \Delta_t$ intersects an annulus $\Omega_n \setminus \Omega_{n - 1}$, $n \in \IN$, then due to property \eqref{iv} of the dyadic decomposition
\begin{align}
\label{Eq: Distance of cubes intersecting shell}
 \dist(x , R) \geq \dist(x , \Omega_{n + 1} \setminus \Omega_{n - 2}) \geq (n - 1) C_1 \delta^k \geq (n - 1) \delta^{-1} C_1 t.
\end{align}
It readily follows from Assumption~\ref{Ass: Geometric setup} that there exists $c>0$ such that $\abs{\Omega \cap B(x,r)} \geq c r^d$ holds for all $x \in \Omega$ and all $r \in (0, a_0)$, where $a_0 > 0$ is given by Theorem~\ref{Thm: Dyadic decomposition}. Properties \eqref{iv} and \eqref{v} of the dyadic decomposition yield
\begin{align}
\label{Eq: Number of cubes intersecting shell}
 \# \big \{R \in \Delta_t: R \cap (\Omega_n \setminus \Omega_{n-1}) \neq \emptyset \big \} \leq \frac{\abs{\Omega_{n+1}}}{c (a_0 \delta^k)^d} \leq \frac{C_1^d(n+2)^d}{c a_0^d} \qquad (n \in \IN_0).
\end{align}
Now, rearrange the cubes in $\Delta_t$ according to the first annulus that they intersect to find
\begin{align*}
 \sum_{R \in \Delta_t} \Big \langle \frac{\dist(x , R)}{t} \Big \rangle^{- M}
\leq \sum_{n = 0}^\infty \frac{C_1^d(n+2)^d}{c a_0^d} (1+ (n-1) \delta^{-1}C_1)^{- M} =: c_M < \infty
\end{align*}
thanks to $M > d+1$.

The second claim is very similar. Choose an arbitrary $x \in Q$ and define $\Omega_n$, $n \geq -2$, as before. By \eqref{Eq: Number of cubes intersecting shell} there are at most $\frac{C_1^d(n+2)^d}{c a_0^d}$ cubes $R \in \Delta_t$ intersecting an annulus $\Omega_n \setminus \Omega_{n-1}$, $n \in \IN_0$, and if this happens then by assumption on $F$, property \eqref{iv} of the dyadic decomposition, and \eqref{Eq: Distance of cubes intersecting shell},
\begin{align*}
 \dist(Q , R \cap F) \geq \max \{\dist(Q , R), \dist(Q,F)\} \geq \max\{(n - 2) \delta^{-1} C_1 t, l t \}.
\end{align*}
Hence, the left-hand side of the estimate in question is bounded by
\begin{align*}
\frac{C_1^d}{c a_0^d} \sum_{n = 0}^{l+2} (n+2)^d \Big(1 + \frac{lt}{s} \Big)^{- M} + \frac{C_1^d}{c a_0^d} \sum_{n = l+3}^\infty (n+2)^d \Big (\frac{(n-2) \delta^{-1} C_1 t}{s} \Big)^{- M}.
\end{align*}
The second sum is controlled by a generic multiple of $s^M t^{-M}$ and so is the first one if $l > 0$.
\end{proof}

A consequence of the preceding lemma is the following. Take $w \in \IC^N$ and regard it as a constant function on $\Omega$. Also fix $s \in (0,1]$. If $Q \in \Delta_t$ for some $t \in ( 0 , 1]$ then Proposition~\ref{Prop: Off-diagonal estimates} and the second part of Lemma \ref{Lem: Sums related to off-diagonal estimates} assure
\begin{align*}
 \sum_{R \in \Delta_t} \|\ind_Q \Theta_s^B (\ind_R w)\| \lesssim \sum_{R \in \Delta_t} \Big\langle \frac{\dist(Q , R)}{s} \Big\rangle^{-(d + 2)} \| \ind_R w\| < \infty.
\end{align*}
As the measure of each cube $Q \in \Delta_t$ is comparable to $t^d$, cf.\ Remark~\ref{Rem: Remark to dyadic decomposition}, each bounded subset of $\Omega$ is covered up to a set of measure zero by finitely many cubes $Q \in \Delta_t$. Now, define $\Theta_s^B w \in \L^2_{\mathrm{loc}}(\Omega ; \IC^N)$ by setting it equal to $\sum_{R \in \Delta_{t}} \ind_Q \Theta_s^B (\ind_R w)$ on each $Q \in \Delta_t$. This definition is independent of the particular choice of $t$. Indeed, if $0 < t_1 < t_2 \leq 1$ and $Q_1 \in \Delta_{t_1}$ is a subcube of $Q_2 \in \Delta_{t_2}$ then 
\begin{align*}
 \ind_{Q_1} \sum_{R_2 \in \Delta_{t_2}} \ind_{Q_2} \Theta_s^B (\ind_{R_2} w) = \sum_{R_2 \in \Delta_{t_2}} \sum_{\substack{R_1 \in \Delta_{t_1} \\ R_1 \subseteq R_2}} \ind_{Q_1} \Theta_s^B (\ind_{R_1} w) = \sum_{R_1 \in \Delta_{t_1}} \ind_{Q_1} \Theta_s^B (\ind_{R_1} w)
\end{align*}
by properties \eqref{i}, \eqref{ii}, and \eqref{iii} of the dyadic decomposition. This gives rise to the following definition.

\begin{definition}
\label{Def: Principle part}
Let $0 < t \leq 1$. The \emph{principal part} of $\Theta_t^B$ is defined as
\begin{align*}
 \gamma_t : \Omega \rightarrow \Lop(\IC^N), \quad \gamma_t(x): w \mapsto (\Theta_t^B w)(x).
\end{align*}
\end{definition}

\begin{remark}
\label{Rem: Principal part is complicated because Omega is not bounded}
If $\Omega$ is bounded then $\HH$ contains the constant $\IC^N$ valued functions and the direct definition of $\Theta_t^B w$ for $t \in (0,1]$ and $w \in \IC^N$ coincides with the one above.
\end{remark}

Next, we introduce the dyadic averaging operator.

\begin{proposition}
\label{Prop: The dyadic averaging operator}
Let $t \in ( 0 , 1 ]$. The \emph{dyadic averaging operator} $A_t$, defined for $u \in \HH$ by
\begin{align*}
A_t u(x) := \fint_{Q(x , t)} u(y) \; \d y  \qquad (x \in \Omega \setminus \NN),
\end{align*}
where $Q(x , t)$ is uniquely characterized by $x \in Q(x , t) \in \Delta_t$, is a contraction on $\HH$.
\end{proposition}

\begin{proof}
Split $\Omega \setminus \NN$ into the dyadic cubes $\Delta_t$ and apply Jensen's inequality to find
\begin{align*}
 \|A_t u \|^2 = \sum_{Q \in \Delta_t} \int_Q \abs{A_t u}^2 \; \d y = \sum_{Q \in \Delta_t} \abs{Q} \bigg|\fint_Q u \; \d y \bigg|^2 \leq \sum_{Q \in \Delta_t} \abs{Q} \fint_Q \abs{u}^2 \; \d y = \|u\|^2. &\qedhere
\end{align*}
\end{proof}

\begin{lemma}
 \label{Lem: Meanvalue estimates for the principle part}
Let $t \in ( 0 , 1 ]$. The operator $\gamma_t A_t: \HH \to \HH$ acting via $(\gamma_t A_t u)(x) = \gamma_t(x) (A_t u)(x)$ is bounded with operator norm uniformly bounded in $t$. Moreover,
\begin{align*}
\fint_Q \| \gamma_t (x)\|_{\Lop(\IC^N)}^2 \; \d x \lesssim 1 \qquad (Q \in \Delta_t)
\end{align*}
with an implicit constant independent of $t$.
\end{lemma}

\begin{proof} The first claim follows straightforwardly from the second one, cf.\ also \cite[Cor.\ 5.4]{Morris}. To prove the second claim fix $Q \in \Delta_t$. With $\{e_j\}_{j=1}^N$ the standard unit vectors in $\IC^N$,
\begin{align*}
 \bigg( \int_Q \| \gamma_t(x) \|_{\Lop(\IC^N)}^2 \; \d x \bigg)^{1/2}
&\lesssim \sum_{j = 1}^N \bigg( \int_Q \abs{\gamma_t(x)e_j}^2 \; \d x \bigg)^{1/2}
\leq \sum_{j = 1}^N \sum_{R \in \Delta_t} \bigg( \int_Q \lvert ( \Theta_t^B (\ind_R e_j )) (x) \rvert^2 \; \d x \bigg)^{1/2}.
\intertext{Proposition~\ref{Prop: Off-diagonal estimates}, item (i) of Remark~\ref{Rem: Remark to dyadic decomposition}, and Lemma \ref{Lem: Sums related to off-diagonal estimates} yield}
&\lesssim \sum_{j = 1}^N \sum_{R \in \Delta_t} \Big\langle \frac{\dist(R , Q)}{t} \Big\rangle^{-(d + 2)} \abs{Q}^{1/2}
\lesssim \abs{Q}^{1/2}
\end{align*}
uniformly in $t$.
\end{proof}

For $u \in \Rg(\g)$ integration over $t \in (0,1]$ on the left-hand side of \eqref{Eq: The quadratic estimate} is now split as
\begin{align}
\label{Eq: The remaining terms}
\begin{split}
\int_0^1 \| \Theta_t^B P_t u \|^2 \; \frac{\d t}{t} &\lesssim \int_0^1 \| (\Theta_t^B - \gamma_t A_t) P_t u \|^2 \; \frac{\d t}{t} \\
&\quad + \int _0^1 \| \gamma_t A_t (P_t - 1) u \|^2 \; \frac{\d t}{t} + \int_0^1 \int_\Omega \prmeas \abs{A_tu(x)}^2 \; \boxmeas.
\end{split}
\end{align}

The idea behind is to compensate the non-integrable singularity at $t=0$ as follows: In the first term $\Theta_t^B P_t u$ is compared with its averages over dyadic cubes. Letting $t \to 0$, the difference is expected to vanish since the diameter of the cubes used for the averaging shrinks to zero. In the second term $P_t$ is compared with the identity operator, which is the strong limit of $P_t$ as $t \to 0$. Finally, the third and most difficult term cries for a Carleson measure estimate. At the beginning of this section we have seen that it remains to bound each of the three terms on the right-hand side by a generic multiple of $\|u\|^2$. This will be done in the remaining sections.

\section{The Proof of Theorem~\ref{Thm: The pibe theorem}: Principal Part Approximation}
\label{Sec: Principal Part Approximation}

\noindent This section is concerned with estimating the first two terms on the right-hand side of \eqref{Eq: The remaining terms}. To start with, recall the classical Poincar\'{e} inequality as it can be deduced from Lemmas 7.12 and 7.16 in \cite{Gilbarg-Trudinger}. Throughout, $u_S:= \fint_S u \; \d x$ is the \emph{mean value} of an integrable function $u: S \to \IC^n$ over a set $S \subseteq \IR^d$ with Lebesgue measure $\abs{S} > 0$.

\begin{lemma}[Poincar\'{e} inequality]
\label{Lem: Convex Poincare inequality}
Let $\Xi \subseteq \IR^d$ be bounded and convex, and let $S$ be a Borel subset of $\Xi$ with $\abs{S} > 0$. Then
\begin{align*}
 \|u - u_S\|_{\L^2(\Xi; \IC)} \leq \frac{(\diam \Xi)^d \abs{B(0,1)}^{1-1/d} \abs{\Xi}^{1/d}}{\abs{S}} \|\nabla u\|_{\L^2(\Xi; \IC^d)} \qquad(u \in \H^1(\Xi; \IC)).
\end{align*}
\end{lemma}

The following weighted Poincar\'{e} inequality is the key to handle the first term in \eqref{Eq: The remaining terms}.

\begin{proposition}[A weighted Poincar\'{e} inequality]
\label{Prop: Weighted Poincare inequality}
For each $M > 2d+2$ there exists $C_M > 0$ such that
\begin{align*}
 \int_{\IR^d} \abs{u(x)-u_Q}^2 \; \Big \langle \frac{\dist(x,Q)}{t} \Big \rangle^{-M} \d x \leq C_M \int_{\IR^d} \abs{t \nabla u(x)}^2 \;  \Big \langle \frac{\dist(x,Q)}{t} \Big \rangle^{2d+2-M} \d x
\end{align*}
holds for all $t \in (0,1]$, all $Q \in \Delta_t$, and all $u \in \H^1(\IR^d; \IC)$.
\end{proposition}

\begin{proof}
Let $t \in (0,1]$ and $Q \in \Delta_t$. Fix some arbitrary $x_0 \in Q$, let $T$ be the affine transformation $x \mapsto x_0 - t^{-1}x$, and put $S:= T(Q)$. Upon replacing $u$ by $u \circ T^{-1}$ it suffices to prove 
\begin{align}
\label{Eq: Scaled weighted poincare}
 \int_{\IR^d} \abs{u(x)-u_S}^2 \; \langle \dist(x,S) \rangle^{-M} \d x \lesssim \int_{\IR^d} \abs{\nabla u(x)}^2 \;  \langle \dist(x,S)\rangle^{2d+2-M} \d x
\end{align}
for arbitrary $u \in \H^1(\IR^d; \IC)$ and an implicit constant independent of $t$, $Q$, and $u$.

Let $C_1$ and $\delta$ be given by Theorem~\ref{Thm: Dyadic decomposition}. Due to property \eqref{iv} of the dyadic decomposition, $S \subseteq B(0, C_1 \delta^{-1})$ and $\abs{S} \simeq 1$. Hence, for $r \geq C_1 \delta^{-1}$ Lemma~\ref{Lem: Convex Poincare inequality} applies with $\Xi = B(0,r)$ and $S$ as above yielding
\begin{align*}
 \int_{\IR^d} \abs{u(x) - u_S}^2 \ind_{B(0,r)}(x) \; \d x \lesssim r^{2d+2} \int_{\IR^d} \abs{\nabla u(x)}^2 \ind_{B(0,r)}(x) \; \d x
\end{align*}
with an implicit constant independent of $u$ and $r$. Integration with respect to $r^{-M-1} \d r$ gives
\begin{align*}
 \int_{\IR^d} \abs{u(x) - u_S}^2 \int_{C_1 \delta^{-1}}^\infty  \ind_{B(0,r)}(x) r^{-M-1} \; \d r \; \d x \lesssim \int_{\IR^d} \abs{\nabla u(x)}^2 \int_{C_1 \delta^{-1}}^\infty r^{2d+1-M} \ind_{B(0,r)}(x) \; \d r \; \d x.
\end{align*}
For fixed $x \in \IR^d$ the inner integrand becomes unequal to $0$ precisely when $r$ gets larger than $\max\{ \abs{x}, C_1 \delta^{-1} \}$ and it is straightforward to verify (draw a sketch!) that
\begin{align*}
 \frac{C_1 \delta^{-1}}{1+ C_1 \delta^{-1}} \cdot (1+\dist(x,S)) \leq \max\{ \abs{x}, C_1 \delta^{-1} \} \leq (1 + C_1 \delta^{-1})(1 +\dist(x,S)).
\end{align*}
Thus, \eqref{Eq: Scaled weighted poincare} follows from the previous estimate by a simple computation of the inner integrals.
\end{proof}

Now, we are in position to estimate the first term in \eqref{Eq: The remaining terms}.

\begin{proposition}[First term estimate]
\label{Prop: First termn estimate}
It holds
\begin{align*}
 \int_0^1 \|(\Theta_t^B - \gamma_t A_t)P_t u\|^2 \; \frac{\d t}{t} \lesssim \|u\|^2 \qquad (u \in \Rg(\g)).
\end{align*}
\end{proposition}

\begin{proof}
We first inspect the integrand $\|(\Theta_t^B - \gamma_t A_t)v\|^2$ for arbitrary $t \in (0,1]$ and $v \in \V^{k}$. Split $\Omega$ into dyadic cubes $Q \in \Delta_t$ and decompose $v = \sum_{R \in \Delta_t} \ind_R v$ to find by the definitions of the principal part and the dyadic averaging operator
\begin{align*}
 \|(\Theta_t^B - \gamma_t A_t)v\|^2 
&= \sum_{Q \in \Delta_t} \Big \| \sum_{R \in \Delta_t} \ind_Q \Theta_t^B(\ind_R v- \ind_R v_Q) \Big \|^2. \\
\intertext{Off-diagonal estimates as in Proposition~\ref{Prop: Off-diagonal estimates} yield}
&\lesssim \sum_{Q \in \Delta_t} \bigg \{\sum_{R \in \Delta_t} \Big \langle \frac{\dist(R,Q)}{t} \Big \rangle^{-3d-4} \|\ind_R (v - v_Q)\| \bigg\}^2 \\
\intertext{and by the Cauchy-Schwarz inequality and Lemma~\ref{Lem: Sums related to off-diagonal estimates},}
&\lesssim \sum_{Q \in \Delta_t} \sum_{R \in \Delta_t} \Big \langle \frac{\dist(R,Q)}{t} \Big \rangle^{-3d-4} \|\ind_R(v - v_Q)\|^2.
\intertext{If $Q, R \in \Delta_t$ and $x \in R$ then $\dist(x,Q) \leq \dist(R,Q) + C_1 \delta^{-1} t$ as follows immediately from property \eqref{iv} of the dyadic decomposition. Consequently,}
&\lesssim \sum_{Q \in \Delta_t} \sum_{R \in \Delta_t} \int_R \abs{v(x) - v_Q}^2 \Big \langle \frac{\dist(x,Q)}{t} \Big \rangle^{-3d-4} \; \d x \\
& = \sum_{Q \in \Delta_t} \int_{\Omega} \abs{v(x) - v_Q}^2 \Big \langle \frac{\dist(x,Q)}{t} \Big \rangle^{-3d-4} \; \d x.
\intertext{Now, use \eqref{V} of Assumption~\ref{Ass: Geometric setup} coordinatewise to construct an extension $\E v \in \H^1(\IR^d; \IC^m)^{k}$ of $v$ to which Proposition~\ref{Prop: Weighted Poincare inequality} applies coordinatewise. Switching sum and integral then leads to}
&\leq \int_{\IR^d} \abs{t \nabla (\E v)(x)}^2 \sum_{Q \in \Delta_t} \Big \langle \frac{\dist(x,Q)}{t} \Big \rangle^{-d-2} \; \d x
\lesssim t^2 \|v\|_{\V^{k}}^2,
\end{align*}
the second step being due to Lemma~\ref{Lem: Sums related to off-diagonal estimates} and the boundedness of $\E: \V^{k} \to \H^1(\IR^d; \IC^m)^{k}$.

On the other hand, Lemmas~\ref{Lem: Basic properties of resolvent operators} and \ref{Lem: Meanvalue estimates for the principle part} bound $\|\Theta_t^B - \gamma_t A_t \|_{\Lop(\HH)}$ uniformly in $t \in (0,1]$. Invoking \eqref{H7}, complex interpolation with the previous estimate yields
\begin{align*}
 \|(\Theta_t^B - \gamma_t A_t)v\|^2 \lesssim t^{2\beta_2} \|v\|_{[\HH, \V^k]_{\beta_2}}^2 \lesssim \|(t^2 \Pi^2)^{\beta_2/2}v\|^2
\end{align*}
for all $v \in \Rg(\g) \cap \dom(\Pi^2)$ and all $t \in (0,1]$. In particular, if $u \in \Rg(\g)$, then due to \eqref{Eq: Pt commutes with g} the previous estimate applies to $v = P_t u$. Hence,
\begin{align*}
 \int_0^1 \|(\Theta_t^B - \gamma_t A_t) P_t u\|^2 \; \frac{\d t}{t} \lesssim \int_0^1 \|(t^2 \Pi^2)^{\beta_2/2}P_t u\|^2 \; \frac{\d t}{t} = \int_0^1 \|\Phi_t(\Pi) u\|^2 \; \frac{\d t}{t}
\end{align*}
with regularly decaying holomorphic functions $\Phi_t: = (t^2 \z^2)^{\beta_2/2} (1+ t^2 \z^2)^{-1}$. Now the conclusion follows by the Schur estimate presented in Remark~\ref{Rem: Self-adjoint operators have quadratic estimates}: Indeed, as in the proof of Proposition~\ref{Prop: Quadratic estimate implies bounded Hinfty calculus} a direct estimate yields some $\zeta \in \L^1(0,\infty; \d r/ r)$ such that $\|\Phi_t(\Pi)Q_s\|_{\Lop(\HH)} \leq \zeta(t s^{-1})$ for all $s,t > 0$ and moreover $\cl{\Rg(\g)} \subseteq \cl{\Rg(\Pi)}$ holds by the unperturbed counterpart of \eqref{Eq: Kernel and range decompositions}.
\end{proof}

\begin{remark}
\label{Rem: Poincare is needed only on whole space}
In contrast to \cite{AKM} we do not require a weighted Poincar\'{e} inequality on $\Omega$ to handle the first term on the right-hand side of \eqref{Eq: The remaining terms}. This is a key observation in order to dispense with smooth local coordinate charts around $\bd \Omega$.
\end{remark}

We head toward the second term in \eqref{Eq: The remaining terms}. The key ingredient is the following interpolation inequality for the unperturbed operators $\g$, $\g^*$, and $\Pi$. The proof follows the one of \cite[Lem.~6]{AKM} line by line except that one invokes Corollary~\ref{Cor: Estimate for complete boundary strip} to estimate the measure of inner boundary strips of dyadic cubes. This results in an exponent $\eta$ as in Corollary~\ref{Cor: Estimate for complete boundary strip} instead of $\eta = 1$ in \cite[Lem.~6]{AKM}.

\begin{lemma}
\label{Lem: Palmenlemma}
If $\Upsilon$ is either of the operators $\g$, $\g^*$ or $\Pi$ then with $\eta >0$ given by Corollary~\ref{Cor: Estimate for complete boundary strip},
\begin{align*}
 \bigg| \fint_Q \Upsilon u \; \d x\bigg|^2 \lesssim \frac{1}{t^\eta} \bigg(\fint_Q \abs{u}^2 \; \d x\bigg)^{\eta/2} \bigg(\fint_Q \abs{\Upsilon u}^2 \; \d x\bigg)^{1 - \eta/2} + \fint_Q \abs{u}^2 \; \d x 
\end{align*}
holds for all $t \in (0,1]$, all $Q \in \Delta_t$, and all $u \in \dom(\Upsilon)$.
\end{lemma}

\begin{proposition}[Second term estimate]
\label{Prop: Second term estimate}
It holds
\begin{align*}
 \int_0^1 \|\gamma_tA_t(P_t - 1)u\|^2 \; \frac{\d t}{t} \lesssim \|u\|^2 \qquad (u \in \HH).
\end{align*}
\end{proposition}

\begin{proof}
Since $A_t$ is a dyadic averaging operator, $A_t^2 = A_t$. Lemma~\ref{Lem: Meanvalue estimates for the principle part} bounds $\|\gamma_t A_t\|_{\Lop(\HH)}$ uniformly in $t \in (0,1]$ so that in fact it suffices to establish
\begin{align*}
 \int_0^1 \|A_t(P_t - 1)u\|^2 \; \frac{\d t}{t} \lesssim \|u\|^2 \qquad(u \in \HH).
\end{align*}
This is certainly true for $u \in \Ke(\Pi)$ since then $P_t u = u$ holds for all $t \in \IR$. Since $\Pi$ is bisectorial, $\HH = \Ke(\Pi) \oplus \cl{\Rg(\Pi)}$. Whence, it remains to consider $u \in \cl{\Rg(\Pi)}$. In this case the conclusion follows by the Schur estimate presented in Remark~\ref{Rem: Self-adjoint operators have quadratic estimates} applied to $T_t := A_t(P_t - 1)$ if $t \leq 1$ and $T_t := 0$ if $t > 1$, provided that we can find some $\zeta \in \L^1(0,\infty; \d r/r)$ such that
\begin{align*}
 \|A_t(P_t - 1)Q_s\|_{\Lop(\HH)} \lesssim  \zeta(ts^{-1}) \qquad (t \in (0,1],\, s>0).
\end{align*}

In fact one can choose $\zeta(r) := \min\{r , r^{-1} + r^{-\eta}\}$. We skip details, since the argument relying on Lemma~\ref{Lem: Palmenlemma}, Lemma \ref{Lem: Sums related to off-diagonal estimates}, and off-diagonal estimates for $P_s$ and $Q_s$ is the same as in \cite[Prop.\ ~5]{AKM}. Note that $\eta = 1$ in \cite{AKM} and that Proposition~\ref{Prop: Off-diagonal estimates} holds for the unperturbed operators $P_s$ and $Q_s$, since if $\{\g, B_1, B_2\}$ satisfies \eqref{H1} - \eqref{H7}, then so does $\{\g, \Id, \Id\}$.
\end{proof}

\section{The Proof of Theorem~\ref{Thm: The pibe theorem}: Principal Part Estimate}
\label{Sec: Principal part estimate}

\noindent After all it remains to estimate the last term in \eqref{Eq: The remaining terms} appropriately, that is to establish
\begin{align}
\label{Eq: Goal of principal part section}
 \int_0^1 \int_\Omega \|\gamma_t(x)\|_{\Lop(\IC^N)}^2 \abs{A_tu(x)}^2 \; \boxmeas \lesssim \|u\|^2 \qquad (u \in \Rg(\Gamma)).
\end{align}
The proof follows the usual strategy of reducing the problem to a Carleson measure estimate, which in turn is established by a $T(b)$ procedure, see e.g.\ \cite{Kato-Square-Root-Proof, AKM, AKM-QuadraticEstimates, Morris, Bandara}. However, since only the last two references deal with the case $\Omega \neq \IR^d$ but under different underlying hypotheses, we include a more detailed argument for our setup.

Recall the notion of a Carleson measure.

\begin{definition}
\label{Def: Carleson measure}
The \emph{Carleson box} $R_Q$ of $Q \in \Delta$ is the Borel set given by $R_Q:= Q \times (0,l(Q)]$. A positive Borel measure $\nu$ on $\Omega \times (0,1]$ satisfying \emph{Carleson's condition}
\begin{align*}
 \|\nu\|_{\mathcal{C}} := \sup_{Q \in \Delta} \frac{\nu(R_Q)}{\abs{Q}} < \infty
\end{align*}
is called \emph{dyadic Carleson measure} on $\Omega \times (0,1]$.
\end{definition}

The following dyadic version of Carleson's theorem can be found in \cite[Thm.\ 4.3]{Morris}.

\begin{theorem}[Carleson, Morris]
\label{Thm: Carleson theorem}
If $\nu$ is a dyadic Carleson measure on $\Omega \times (0,1]$, then
\begin{align*}
 \iint_{\Omega \times (0,1]} \abs{A_tu(x)}^2 \; \d \nu(x,t) \lesssim \|\nu\|_{\mathcal{C}} \|u\|^2 \qquad (u \in \HH).
\end{align*}
\end{theorem}

So, \eqref{Eq: Goal of principal part section} follows if $\prmeas \; \boxmeas$ is a Carleson measure on $\Omega \times (0,1]$ and it is this property of the principal part $\gamma_t$ we are going to establish in the following.

We begin by fixing $\sigma >0$; its value to be chosen later. Also, by compactness, we fix a finite set $\mathcal{F}$ in the boundary of the unit ball of $\Lop(\IC^N)$ such that the sets
\begin{align}
\label{Eq: Pizza slices of matrices}
 K_\nu := \Big \{ \nu' \in \Lop(\IC^N) \setminus \{0\}: \Big\| \frac{\nu'}{\|\nu'\|_{\Lop(\IC^N)}} - \nu \Big\|_{\Lop(\IC^N)} \leq \sigma \Big \} \qquad (\nu \in \mathcal{F})
\end{align}
cover $\Lop(\IC^N) \setminus \{0\}$. By a standard argument using the John-Nierenberg Lemma, the following proposition implies Carleson's condition for the measure $\prmeas \; \boxmeas$, cf.\ e.g.\ \cite[p.\ 906]{Morris}.

\begin{proposition}
\label{Prop: The real hard work}
There exist $\beta, \beta' >0$ such that for each $Q \in \Delta$ and for each $\nu \in \Lop(\IC^N)$ with $\|\nu\|_{\Lop(\IC^N)} = 1$, there is a collection $\{Q_k\}_k \subseteq \Delta$ of pairwise disjoint subcubes of $Q$ such that $\abs{E_{Q,\nu}} > \beta \abs{Q}$, where $E_{Q, \nu}:= Q \setminus \bigcup \{Q_k\}_k$, and such that
\begin{align}
\label{Eq: The real hard work 1}
 \iint_{\substack{(x,t) \in E_{Q,\nu}^*\\ \gamma_t(x) \in K_\nu}} \prmeas \; \boxmeas \leq \beta' \abs{Q},
\end{align}
where $E_{Q, \nu}^* := R_Q \setminus \bigcup \{R_{Q_k}\}_k$.
\end{proposition}

Hence, our task is to prove Proposition~\ref{Prop: The real hard work}. We closely follow \cite[pp.\ ~23-26]{AKM-QuadraticEstimates}. For the proof keep $Q \in \Delta$ and $\nu \in \Lop(\IC^N)$ with $\|\nu\|_{\Lop(\IC^N)} =1$ fixed and put $\tau:= l(Q)$ for brevity. Define the dilated cube $2Q: = \{x \in \IR^d: \dist(x,Q) \leq l(Q)\}$. Since the adjoint matrix $\nu^* \in \Lop(\IC^N)$ has norm $1$ there are $\omega, \hat{\omega} \in \IC^N$ such that
\begin{align}
\label{Eq: Properties of omega vectors}
 \abs{\omega} = \abs{\hat{\omega}} = 1 \quad \text{and} \quad \omega = \nu^* \hat{\omega}.
\end{align}
We prepare for the usual $T(b)$ argument but similar to \cite[Sec.\ 3.6]{Auscher-Rosen-Rule} we use $\ind_{2Q} \omega$ as a test function rather than some smoothened version of it. This leads to a simplification of the argument compared to \cite[Sec.\ 4.4]{AKM}. In the subsequent estimates a constant is called \emph{admissible} if it neither depends on the quantities fixed above nor on $\sigma$ its value still to be chosen. For $\eps > 0$ we then put
\begin{align}
\label{Eq: fQeps}
 f_{Q, \eps}^\omega := (1-\eps \tau \i \g R_{\eps \tau}^B) \ind_{2Q} \omega = \ind_{2Q} \omega - \eps \tau \i \g(1+\eps \tau \i \pb)^{-1} \ind_{2Q} \omega = (1 + \eps \tau \i \gbs)R_{\eps \tau}^B \ind_{2Q} \omega
\end{align}
and derive the following estimates.

\begin{lemma}
\label{Lem: Tb estimates}
There exist admissible constants $A_1, A_2, A_3 > 0$ such that for all $\eps>0$ it holds
\begin{align*}
 \|f_{Q,\eps}^\omega\| \leq A_1 \abs{Q}^{1/2}, \quad \iint_{R_Q} \abs{\Theta_t^B f_{Q,\eps}^\omega(x)}^2 \; \boxmeas \leq \frac{A_2}{\eps^2} \abs{Q},  \quad \bigg| \fint_Q f_{Q,\eps}^\omega(x) \; \d x- \omega \bigg|^2 \leq A_3(\eps^\eta + \eps^2).
\end{align*}
\end{lemma}

\begin{proof}
Note $\abs{2 Q} \leq (1 + C_1)^d l(Q)^d \lesssim \abs{Q}$ by property \eqref{iv} of the dyadic decomposition. Hence, \eqref{Eq: Kernel and range decompositions} and Lemma~\ref{Lem: Basic properties of resolvent operators} yield
\begin{align}
\label{Eq: Tb estimates 1}
\|\g R_{\eps \tau}^B \ind_{2Q} \omega\| + \|\gbs R_{\eps \tau}^B \ind_{2Q} \omega\| 
= (\eps \tau)^{-1} \|(1- R_{\eps \tau}^B) \ind_{2Q} \omega\|
\lesssim (\eps \tau)^{-1} \abs{Q}^{1/2}
\end{align}
with admissible implicit constants. From this, the first estimate follows. For the second estimate check by nilpotence of $\g$ and $\gbs$ that
\begin{align*}
 \Theta_t^B f_{Q,\eps}^\omega = t \gbs P_t^B (1+ \eps \tau \i \gbs)R_{\eps \tau}^B \ind_{2Q} \omega = t P_t^B \gbs R_{\eps \tau}^B \ind_{2Q} \omega.
\end{align*}
Recalling $l(Q) = \tau$, integration gives
\begin{align*}
\iint_{R_Q} \abs{\Theta_t^B f_{Q,\eps}^\omega(x)}^2 \; \boxmeas 
&\leq \int_0^\tau t \| P_t^B \gbs R_{\eps \tau}^B \ind_{2Q} \omega\|^2 \; \d t 
\lesssim \int_0^\tau t \|\gbs R_{\eps \tau}^B \ind_{2Q} \omega\|^2 \; \d t 
\end{align*}
and \eqref{Eq: Tb estimates 1} yields the claim. For the third estimate apply Lemma~\ref{Lem: Palmenlemma} with $\Upsilon = \g$ to find
\begin{align*}
 \bigg| \fint_Q f_{Q,\eps}^\omega \; \d x- \omega \bigg|^2 
&= \bigg| \fint_Q (f_{Q,\eps}^\omega - \ind_{2Q} \omega) \; \d x \bigg|^2
=  (\eps \tau)^2 \bigg| \fint_Q \g R_{\eps \tau}^B \ind_{2Q} \omega \; \d x \bigg|^2 \\
& \lesssim \frac{(\eps \tau)^2}{\tau^\eta} \bigg(\fint_Q \abs{R_{\eps \tau}^B \ind_{2Q} \omega}^2 \; \d x \bigg)^{\eta/2} \cdot \bigg( \fint_Q \abs{\g R_{\eps \tau}^B \ind_{2Q} \omega}^2 \; \d x\bigg)^{1 - \eta/2} \\
&\quad+ (\eps \tau)^2 \fint_Q \abs{R_{\eps \tau}^B \ind_{2Q} \omega}^2 \; \d x.
\intertext{By Lemma~\ref{Lem: Basic properties of resolvent operators} and \eqref{Eq: Tb estimates 1}, keeping in mind $\tau \leq 1$, it follows}
&\lesssim \frac{(\eps \tau)^2}{\tau^\eta} \cdot (\eps \tau)^{\eta - 2} + (\eps \tau)^2 \leq \eps^\eta + \eps^2. \qedhere
\end{align*}
\end{proof}

From now on keep $\eps > 0$ fixed as the solution of $A_3(\eps^\eta + \eps^2)= \frac{1}{2}$ with $A_3$ as in the preceding lemma. We shall simply write $f_Q^\omega$ instead of $f_{Q,\eps}^\omega$. Owing to Lemma~\ref{Lem: Tb estimates} and $\abs{\omega} =1$ we find
\begin{align}
\label{Eq: Mean value integral scalar product}
2 \Re \bigg\langle \omega, \fint_Q f_Q^\omega \; \d x\bigg \rangle = \bigg|\fint_Q f_Q^\omega \; \d x\bigg|^2 + \abs{\omega}^2 -\bigg| \fint_Q f_Q^\omega \; \d x- \omega \bigg|^2  \geq \frac{1}{2}.
\end{align}

The following lemma now follows literally as in \cite[Lem.\ 5.11]{AKM-QuadraticEstimates}.

\begin{lemma}
\label{Lem: Selection of the subcubes}
There exist admissible constants $\beta, \rho > 0$ and a collection $\{Q_k\}_k \subseteq \Delta$ of dyadic subcubes of $Q$ such that $\abs{E_{Q,\nu}} > \beta \abs{Q}$ where $E_{Q, \nu}:= Q \setminus \bigcup \{Q_k\}_k$, and such that
\begin{align}
\label{Eq: Selection of the subcubes 1}
 \Re \bigg \langle \omega, \fint_{Q'} f_Q^\omega(x) \; \d x \bigg \rangle \geq \rho \qquad \text{and} \qquad \fint_{Q'} \abs{f_Q^\omega(x)} \; \d x \leq \frac{1}{\rho}
\end{align}
for all dyadic subcubes $Q' \in \Delta$ of $Q$ which satisfy $R_{Q^{\prime}} \cap E_{Q , \nu}^* \neq \emptyset$, where $E_{Q,\nu}^* := R_Q \setminus \bigcup \{R_{Q_k}\}_k$.
\end{lemma}

Let $\rho$, $\{Q_k\}_k$, $E_{Q, \nu}$, and $E_{Q, \nu}^*$ be as provided by Lemma~\ref{Lem: Selection of the subcubes}. We shall prove the estimates in Proposition~\ref{Prop: The real hard work} for these choices. Eventually, we fix the value of $\sigma > 0$ determining the size of the `pizza slices' $K_\nu$ in \eqref{Eq: Pizza slices of matrices} as $\sigma:= \frac{\rho^2}{2}$. For the next lemma recall that $\NN$ is the exceptional set defined in Remark~\ref{Rem: Remark to dyadic decomposition}.

\begin{lemma}
\label{Lem: Lower bound for principal part}
Suppose $(x,t) \in E_{Q, \nu}^*$ is such that $x \notin \NN$ and $\gamma_t(x) \in K_\nu$. Then
\begin{align*}
 \big|\gamma_t(x)(A_t f_Q^\omega(x))\big| \geq \frac{\rho}{2} \|\gamma_t(x)\|_{\Lop(\IC^N)}.
\end{align*}
\end{lemma}

\begin{proof}
Due to $x \notin \NN$ there exists a unique $Q' \in \Delta_t$ that contains $x$. Hence $R_{Q'} \cap E_{Q,\nu}^* \neq \emptyset$. Since by definition $A_t f_Q^\omega(x) = \fint_{Q'} f_Q^\omega(y) \; \d y$, the previous lemma and the relations between $\nu$, $\omega$, and $\hat{\omega}$, cf.\ \eqref{Eq: Properties of omega vectors}, yield
\begin{align*}
 \big|\nu (A_t f_Q^\omega(x))\big| \geq \Re \big \langle \hat{\omega}, \nu (A_t f_Q^\omega(x))\big\rangle = \Re \big\langle\omega, A_t f_Q^\omega(x)\big\rangle \geq \rho
\end{align*}
and furthermore -- due to $\gamma_t(x) \in K_\nu$ -- also
\begin{align*}
 \bigg| \frac{\gamma_t(x)}{\|\gamma_t(x)\|_{\Lop(\IC^N)}} (A_t f_Q^\omega(x))\bigg|
\geq \abs{\nu (A_t f_Q^\omega(x))} - \abs{A_t f_Q^\omega(x)} \cdot \bigg\|\frac{\gamma_t(x)}{\|\gamma_t(x)\|_{\Lop(\IC^N)}} - \nu\bigg\|_{\Lop(\IC^N)} \geq \frac{\rho}{2}. &\qedhere
\end{align*}
\end{proof}

Finally we complete the proof of Proposition~\ref{Prop: The real hard work}.

\begin{proof}[Proof of Proposition~\ref{Prop: The real hard work}]
 It remains to establish \eqref{Eq: The real hard work 1}. The crucial observation is that Lemma~\ref{Lem: Lower bound for principal part} allows to reintroduce the dyadic averaging operator:
\begin{align*}
 \iint_{\substack{(x,t) \in E_{Q,\nu}^*\\ \gamma_t(x) \in K_\nu}} \prmeas \; \boxmeas
&\lesssim \iint_{R_Q} \abs{\gamma_t(x)(A_tf_Q^\omega(x))}^2 \; \boxmeas \\
&\leq 2 \iint_{R_Q} \abs{\Theta_t^Bf_Q^\omega}^2 \; \boxmeas + 2 \iint_{R_Q} \abs{(\Theta_t^B - \gamma_tA_t)f_Q^\omega}^2 \; \boxmeas.
\end{align*}
Lemma~\ref{Lem: Tb estimates} bounds the first term on the right-hand side by $2 A_2 \eps^{-2} \abs{Q}$. To handle the second one put $u:= \eps \tau \i \g R_{\eps \tau}^B \ind_{2Q} \omega \in \Rg(\g)$. Then due to $f_Q^\omega = \ind_{2Q} \omega - u$, see \eqref{Eq: fQeps}, it remains to show
\begin{align}
\label{Eq: Final goal of hard work}
 \int_0^{\tau} \| \ind_Q(\Theta_t^B - \gamma_tA_t) \ind_{2Q} \omega\|^2 \; \frac{\d t}{t} +  \int_0^{\tau} \| \ind_Q(\Theta_t^B - \gamma_tA_t)u\|^2 \; \frac{\d t}{t} \lesssim \abs{Q}.
\end{align}

For the first term on the left-hand side note $A_t \ind_{2Q} \omega(x) = \omega$ for all $x \in Q$ and $t \in (0, \tau)$ so that by definition of the principal part
\begin{align*}
\|\ind_Q(\Theta_t^B \ind_{2Q} \omega - \gamma_t A_t \ind_{2Q} \omega)\| 
&\leq \sum_{R \in \Delta_{\tau}} \|\ind_Q \Theta_t^B ( \ind_{R \cap (\IR^d \setminus 2Q)} \omega )\|.
\intertext{Proposition \ref{Prop: Off-diagonal estimates} gives}
&\lesssim \sum_{R \in \Delta_\tau} \Big \langle \frac{\d(Q, R \cap (\IR^d \setminus 2Q))}{t} \Big \rangle^{-(d+2)} \|\ind_{R \cap (\IR^d \setminus 2Q)} \omega\|.
\end{align*}
Since dyadic cubes of the same step size are comparable in measure, we get for each $R \in \Delta_\tau$ that $\|\ind_{(\IR^d \setminus 2Q) \cap R} \omega\| \leq \abs{R}^{1/2} \simeq \abs{Q}^{1/2}$. Now, the latter sum is under control by the second part of Lemma~\ref{Lem: Sums related to off-diagonal estimates} with $l = 1$. Altogether,
\begin{align*}
\|\ind_Q(\Theta_t^B \ind_{2Q} \omega - \gamma_t A_t \ind_{2Q} \omega)\| 
\lesssim \abs{Q}^{1/2}  \frac{t^{d+2}}{\tau^{d+2}}.
\end{align*}
Going back to \eqref{Eq: Final goal of hard work}, this gives the required bound for the first term. The second one is bounded by
\begin{align*}
\int_0^1 \|\Theta_t^B(1-P_t)u\|^2 + \|(\Theta_t^B - \gamma_t A_t)P_t u\|^2 + \|\gamma_t A_t (P_t - 1)u\|^2 \; \frac{\d t}{t}
\end{align*}
and these three terms have already been taken care of in \eqref{Eq: Abstract setting quadratic estimate} and Propositions~\ref{Prop: First termn estimate} and \ref{Prop: Second term estimate} bounding them by a multiple of $\|u\|^2$. However, in view of \eqref{Eq: Tb estimates 1} we find $\|u\|^2 \lesssim \abs{Q}$. This completes the proof of Proposition \ref{Prop: The real hard work}. 
\end{proof}
\subsection*{Acknowledgments}
The authors want to thank P.\ Auscher for helping to improve this manuscript in several ways and J.\ Rozendaal for valuable discussions on functional calculi.

\end{document}